\documentclass[11pt]{amsart}

\usepackage{amssymb}
\usepackage{dsfont}
\usepackage{color}
\usepackage{mathrsfs}

\usepackage{enumerate}

\usepackage{graphicx}

\makeatletter
\@namedef{subjclassname@2010}{%
  \textup{2010} Mathematics Subject Classification}
\makeatother

\usepackage[T1]{fontenc}

\newtheorem{thm}[equation]{Theorem}
\newtheorem{cor}[equation]{Corollary}
\newtheorem{lem}[equation]{Lemma}
\newtheorem{prop}[equation]{Proposition}

\newtheorem{fac}[equation]{Fact}

\theoremstyle{definition}

\newtheorem{rem}[equation]{Remark}
\newtheorem{exa}[equation]{Example}

\numberwithin{equation}{section}

\newcommand{\rd}{\mathds{R}^{d}}
\newcommand{\rn}{\mathds{R}^{n}}
\newcommand{\nd}{\mathds{N}^{d}}
\newcommand{\nn}{\mathds{N}^{n}}

\newcommand{\Op}{\operatorname{Op}}
\newcommand{\sgn}{\operatorname{sgn}}

\newcommand{\spann}{\operatorname{span}}

\newcommand{\Rez}{\operatorname{Re}}

\newcommand{\loc}{\operatorname{loc}}

\newcommand{\srn}{\mathcal{S}(\mathds{R}^{n})}
\newcommand{\srd}{\mathcal{S}(\mathds{R}^{d})}

\newcommand{\sco}{\mathcal{S}(\mathfrak{h})}

\newcommand{\scop}{\mathcal{S}'(\mathfrak{h})}

\newcommand{\smgh}{S(m,g;\mathfrak{h}^{*})}
\newcommand{\smgd}{S(m,g;\mathfrak{h})}
\newcommand{\smzgh}{S(m_{|z|},g;\mathfrak{h}^{*})}
\newcommand{\smzgd}{S(m_{|z|},g;\mathfrak{h})}
\newcommand{\smghp}{S(m',g;\mathfrak{h}^{*})}
\newcommand{\smghpp}{S(mm',g;\mathfrak{h}^{*})}
\newcommand{\smigh}{S(m^{-1},g;\mathfrak{h}^{*})}
\newcommand{\smigd}{S(m^{-1},g;\mathfrak{h})}
\newcommand{\smzigh}{S(m_{|z|}^{-1},g;\mathfrak{h}^{*})}
\newcommand{\smzigd}{S(m_{|z|}^{-1},g;\mathfrak{h})}
\newcommand{\sogh}{S(1,g;\mathfrak{h}^{*})}
\newcommand{\sogd}{S(1,g;\mathfrak{h})}

\newcommand{\hmgh}{H(m,g;\mathfrak{h})}
\newcommand{\lth}{L^{2}(\mathfrak{h})}
\newcommand{\ltrn}{L^{2}(\mathds{R}^{n})}

\begin{document}


\baselineskip=17pt



\title[Semigroups of measures on the Heisenberg group]{Symbolic calculus\\ and convolution semigroups of measures\\ on the Heisenberg group}
\author[K. Beka{\l}a]{Krystian Beka{\l}a}
\address{Institute of Mathematics
\\ University of Wroc{\l}aw
\\ 50-384 Wroc{\l}aw, Poland}
\email{krystian.bekala@math.uni.wroc.pl}

\date{}

\begin{abstract}
Let $P$ be a generalized laplacian on $\mathds{R}^{2n+1}$. It is known that $P$ is the generating functional of semigroups of measures $\mu_{t}$ on the Heisenberg group $\mathds{H}^{n}$ and $\nu_{t}$ on the Abelian group $\mathds{R}^{2n+1}$. Under some smoothness and growth conditions on the functional $P$ expressed in terms of its Abelian Fourier transform $\widehat{P}$
we show that the semigroup $\mu_{t}$ is a kind of perturbation of the semigroup $\nu_{t}$. More precisely, we give pointwise estimates for the difference of the densities of the measures $\mu_{t}$ and $\nu_{t}$.

As a consequence we get a description of the asymptotic behavior at the origin or pointwise estimates for the densities of the semigroup of measures on the Heisenberg group which is an analogue (via generating functional) of the symmetrized gamma (gamma-variance) semigroup on $\mathds{R}^{2n+1}$. 

The main tool is a symbolic calculus for convolution operators on the Heisenberg group.
\end{abstract}

\subjclass[2010]{Primary 22E25; Secondary 43A30}

\keywords{semigroups of measures, Heisenberg group, symbolic calculus, generalized laplacians, gamma-variance semigroup}

\maketitle

\section{Statement of the result} \label{s1}
We work on the Heisenberg group $\mathds{H}^{n}$ that is the manifold $\rn \times \rn \times \mathds{R}$ with the Campbell-Hausdorff multiplication
\begin{equation*}
xy=(x_{1},x_{2},x_{3})(y_{1},y_{2},y_{3})=(x_{1}+y_{1},x_{2}+y_{2},x_{3}+y_{3}+{1 \over 2}(x_{1}\cdot y_{2}-x_{2}\cdot y_{1})).
\end{equation*}
The Heisenberg group is a noncommutative two step nilpotent Lie group.
The convolution of functions (e.g. Schwartz functions)
\begin{equation} \label{splotc}
f*g(x)=\int_{\mathds{R}^{2n+1}} f(xy^{-1})g(y) dy = \int_{\mathds{R}^{2n+1}} f(y)g(y^{-1}x) dy,
\end{equation}
or Borel measures on $\mathds{H}^{n}$ is also noncommutative.

A family $(\mu_{t})_{t>0}$ of probabilistic measures on $\mathds{R}^{2n+1}$ is a \textit{convolution semigroup} of measures on the Heisenberg group if it satisfies
\begin{enumerate}[\upshape (i)]
\item $\mu_{t}*\mu_{s}=\mu_{t+s}, \qquad t,s >0$,
\item $\lim_{t \to 0} \langle \mu_{t}, f \rangle = f(e),  \qquad f \in C_{b}(\mathds{R}^{2n+1})$.
\end{enumerate}
In the similar way one can define a convolution semigroup of measures on the Abelian groups $\rd$ with respect to the ordinary convolution which we denote by $*_{0}$ or, more generally, on Lie groups with the induced convolution.
For such groups there is a rich theory of convolution semigroups of measures which essentially started in the article of Hunt \cite{Hun}.

If $\mu_{t}$ is a convolution semigroup of measures, then
\begin{equation*}
\langle P, f \rangle = \lim_{t \to 0} {\langle \mu_{t}, f \rangle - f(e) \over t}, \qquad f \in C_{c}^{\infty}(G),
\end{equation*}
defines a functional which is called the \textit{generating functional}.
Semigroups of measures are characterized by their generating functionals.
It is also known that generating functionals are exactly generalized laplacians. A distribution $P$ is called a \textit{generalized laplacian}, if it is real and satisfies the following maximum property
\begin{equation} \label{max}
\langle P, f \rangle \leq 0,
\end{equation}
for real functions $f \in C_{c}^{\infty}(G)$ such that $f(e)=\sup_{x \in G}|f(x)|$.
The Hunt theory also gives better continuity properties of generating functionals. In particular they are tempered distributions.

Let us notice that the definition of a generalized laplacian does not depend on the group structure for a given underlying manifold. Consequently, generalized laplacians on the Heisenberg $\mathds{H}^{n}$ are the same as on the Abelian group $\mathds{R}^{2n+1}$.

In our work we consider semigroups whose generating functionals satisfy some conditions of "admissibility". Such conditions are expressed in terms of their Fourier transforms, i.e. associated continuous negative definite functions. An example of an admissible generalized laplacian is
\begin{equation*}
\langle \Gamma, f \rangle = c_{n} \int_{\mathds{R}^{2n+1} \backslash \{0\}} {f(x)-f(0) \over \|x\|^{{2n+1 \over 2}}} K_{{2n+1 \over 2}}(\|x\|) dx, \qquad f \in C_{c}^{\infty}(\mathds{R}^{2n+1}),
\end{equation*}
where $K$ is the modified Bessel function of the second kind. The constant $c_{n}$ is given by $-\widehat{\Gamma}=\log(1+\|\xi\|^{2})$. In the setting of the Abelian group $\mathds{R}^{2n+1}$ the distribution $\Gamma$ is the generating functional of the so-called symmetrized gamma semigroup.
In this work we also present a quite large class of admissible generalized laplacians by using Bernstein functions. 

Let $P$ be the generating functional of a semigroup of measures $\mu_{t}$ on the Heisenberg group and also the generating functional of a semigroup of measures $\nu_{t}$ on the Abelian group $\mathds{R}^{2n+1}$. We prove estimates of the difference of the transforms $\widehat{\mu_{t}}-\widehat{\nu_{t}}$ (and their derivatives) and show that they are small with respect to the space variable $\xi \in \mathds{R}^{2n+1}$ and the time variable $t>0$.
This leads to our main theorem which says that the difference of the measures $\mu_{t}$ and $\nu_{t}$ agrees with a function $k_{t}$ which is smooth outside the origin and for all $t<1$, all $N \in \mathds{N}$ and all $\alpha \in \mathds{N}^{2n+1}$ satisfies
$$|\partial_{x}^{\alpha} k_{t}(x)| \leq \begin{cases}
c_{n,\alpha} t^{2}\|x\|^{-(2n+1)+2-|\alpha|} & \mbox{for } \|x\| \leq 1,\cr
c_{n,\alpha,N} t^{2}\|x\|^{-N} & \mbox{for } \|x\| \geq 1. \end{cases}$$

As an application we get some conditions for the measures $\mu_{t}$ to have densities in $L^{1}$ or in $L^{p}$.
We also apply the above estimates obtaining pointwise estimates for the densities of semigroups of measures on the Heisenberg group.
In particular, for the generalized laplacian $\Gamma$ we deduce the following estimates
\begin{equation*}
v_{t}(x) \leq c_{n} t\|x\|^{-(2n+1)+2t}, \qquad \|x\| \leq 1, \ t<1, 
\end{equation*}
or the asymptotic behavior for $t<1$,
\begin{align*}
v_{t}(x) &\asymp t\|x\|^{-(2n+1)+2t}, \qquad \|x\| \to 0.
\end{align*}

One of the most important tool using in the work is a symbolic calculus for convolution operators on the Heisenberg group which is similar to the Weyl-H{\"o}rmander symbolic calculus of pseudodifferential operators. The idea of the calculus consists in describing the product $a \# b = (a^{\vee} * b^{\vee})^{\wedge}$ for some classes of symbols.

Symbolic calculi for convolution operators and their application to investigate $L^{2}-$ boundedness of such operators were describe in Howe \cite{Hol3} and Melin \cite{Mel}.
A development of a symbolic calculus for invariant operators on homogeneous Lie groups and subtilizing conditions for $L^{2}$-boundedness in terms of symbols one can find in G{\l}owacki 
\cite{Glo1}, \cite{Glo5}. 
In the revised version of Taylor \cite{Tay} there are some operator calculi and the respond symbolic calculi as well as their applications.
Recently, symbolic calculi for (not necesserily invariant) pseudodifferential operators on the Heisenberg group appeared in Bahouri-Fermanian-Kammerer-Gallagher \cite{BKG} and 
Fischer-Ruzhansky \cite{FiRu}.

Symbolic calculi are often "tailor-made" to specified needs. In our case it was important to consider "small" weights, inverse symbols and symbols depending on parameters. None of existing calculi was fully satisfactory for our purposes and we present here a new one which is based on them.

Convolution semigroups of measures on the Heisnberg group have already been studied, also in the more general setting of nilpotent Lie groups.
Gaveau \cite{Gav} and Hulanicki \cite{Hul2} found independently the Fourier transform of the densities for the semigroup of measures generated by the sublaplacian on the Heisenberg group. In that context we also refer to e.g. Folland \cite{Vol} and  Bonfiglioli-Lanconelli-Uguzzoni \cite{blu}.
Barczy-Pap \cite{BaP} studied gaussian measures on the Heisenberg group. They found an explicit form of the (operator-valued) group Fourier transform of gaussian measures and gave conditions for convolution of gaussian measures to be again a gaussian measure.
L{\'e}vy processes on the Heisenberg group was investigated in Applebaum-Cohen \cite{ApC}. The authors described quantized infinitesimal generators and their analytical properties. In particular, they considered so-called phase-dominated processes. They also studied the relationship between processes and Dirichlet forms. 
Stable semigroups of measures on the Heisenberg group was studied by G{\l}owacki \cite{Glo5}.
He showed that such measures have densities in $L^{2}$. 
In setting of nilpotent Lie groups G{\l}owacki i Hebisch \cite{GH} proved pointwise estimates for the densities. It was generalized by Dziuba{\'n}ski \cite{Dzi} who found the asymptotic behavior at infinitity.

The general theory of convolution semigroups of measures on Lie groups we refer to Hunt \cite{Hun}, Hulanicki \cite{Hul}, Duflo \cite{Duf} and Faraut \cite{Far}.


\section{Preliminaries} \label{s2}
\subsection{Notion of Euclidean space}
We consider $d$-dimensional Euclidean space $\rd$ with the standard scalar product $x \cdot y = \sum_{i=1}^{d}x_{i}y_{i}$ and the Euclidean norm $\|x\|=(\sum_{i=1}^{d}x_{i}^{2})^{1 \over 2}$.
We identify $\rd$ and its dual. 

Let $T_{j}f(x)=x_{j}f(x)$. For a multiindex  $\alpha \in \nd$ we denote
\begin{equation*} 
T^{\alpha}f(x)=x_{1}^{\alpha_{1}} \ldots x_{d}^{\alpha_{d}}f(x), \qquad
\partial^{\alpha}f(x)=\partial^{\alpha_{1}}_{1}\ldots\partial^{\alpha_{d}}_{d}f(x).
\end{equation*}
Sometimes we will write $T^{\alpha}_{x}$ or $\partial^{\alpha}_{x}$ to emphasize that the variable $x$ is in use. For $\alpha \in \nd$ let $|\alpha|=\sum_{i=1}^{d} \alpha_{i}$.
Also let
\begin{equation*}
\widetilde{f}(x)=f(-x), \qquad \qquad \tau_{y}f(x)=f(y+x), \qquad y \in \rd.
\end{equation*}

We use subspaces of the space of continuous function $C(\rd)$: functions vanishing at infinity $C_{0}(\rd)$ and bounded functions $C_{b}(\rd)$, both with the supremum norm
$\|f\|_{\infty} = \sup_{x \in \rd} |f(x)|$.
The Schwartz space $\srd$ consists of smooth functions for which the following seminorms
\begin{equation*} 
\|f\|_{(N)} = \max_{|\alpha|+|\beta|\leq N} \|T^{\alpha}\partial^{\beta}f\|_{\infty}, \qquad N \in \mathds{N},
\end{equation*}
are finite. A dense subspace of the Schwartz class $\srd$ is the space of smooth functions with compact support $C_{c}^{\infty}(\rd)$. 

Let $\mathcal{S}'(\rd)$ be the space of tempered distributions, i.e. continuous linear functionals on $\srd$. More general, the classes of functions $C_{c}^{\infty}(\rd) \subset \srd \subset C^{\infty}(\rd)$ led to the classes of distributions $\mathcal{D}'(\rd) \supset \mathcal{S}'(\rd) \supset \mathcal{E}'(\rd)$. The class $\mathcal{E}'(\rd)$ denotes distributions with compact support.
The pairing of a distribution $A$ with a smooth function $f$ is denoted by $\langle A, f \rangle$, whenever it makes sense. Moreover, for $\alpha \in \nd$ we denote
\begin{equation*}
\langle T^{\alpha}A, f \rangle = \langle A, T^{\alpha}f \rangle, \qquad
\langle \partial^{\alpha}A, f \rangle = (-1)^{|\alpha|}\langle A, \partial^{\alpha}f \rangle, \qquad
\langle \widetilde{A}, f \rangle = \langle A, \widetilde{f} \rangle,
\end{equation*}
and $\langle \psi A, f \rangle = \langle A, \psi f \rangle$
for $\psi \in \srd$. The Dirac delta distribution is defined by $\langle \delta_{0}, f \rangle = f(0).$
 
For $f \in \srd$ its \textit{Abelian Fourier transform} $\widehat{f}$ is given by
\begin{equation} \label{teef}
\mathcal{F}f(\xi) = \widehat{f}(\xi)= \int_{\mathds{R}^{d}} e^{-ix \cdot \xi} f(x) dx.
\end{equation} 
The Fourier transform $\mathcal{F}$ is a bijection of the Schwartz class. The Inverse Fourier transform satisfies
\begin{equation*} 
\mathcal{F}^{-1}\widehat{f}(x) =  (2\pi)^{-d} \int_{\mathds{R}^{d}} e^{ix \cdot \xi} \widehat{f}(\xi) d\xi = f(x)
\end{equation*}
and thus
\begin{equation} \label{foin}
f(x)=(2\pi)^{-d} \int_{\rd} \int_{\rd} e^{i(x-y)\cdot\xi}f(y) dy d\xi \qquad f \in \mathcal{S}(\rd).
\end{equation}
We also have
\begin{equation*}
\mathcal{F}T^{\alpha} = (i\partial)^{\alpha}\mathcal{F}, \qquad T^{\alpha}\mathcal{F} = (-i)^{|\alpha|} \mathcal{F}\partial^{\alpha}.
\end{equation*}
The Fourier transform extends by duality to $\mathcal{S}'(\rd)$. 

We denote by $L^{p}(\rd)$ the usual integrable functions with  $p$-th power. We say that $f \in L^{p}_{\loc}(\rd)$, if for every compact $K \subseteq \rd$ we have
\begin{equation*}
\int_{K} |f(x)|^{p} dx < \infty. 
\end{equation*}
A distribution $A \in \mathcal{D}'(\rd)$ agrees with a function $a \in L^{1}_{\loc}(\rd)$, if
\begin{equation*}
\langle A, f \rangle = \int a(x)f(x) dx, \qquad f \in C_{c}^{\infty}(\rd).
\end{equation*}
 
The space $L^{2}(\rd)$ is a Hilbert space with the inner product
\begin{equation*}
\langle f, g\rangle_{L^{2}(\rd)} = \int_{\rd} f(x) \overline{g(x)} dx.
\end{equation*} 
The Fourier transform of Schwartz functions preserves $L^{2}(\rd)$ norm and extends to an isometry of  $L^{2}(\rd)$ which will be denoted in the same way. The identity
\begin{equation} \label{plancz}
\|f\|_{L^{2}(\rd)} = (2\pi)^{-d}\|\widehat{f}\|_{L^{2}(\rd)}, \qquad f \in L^{2}(\rd), 
\end{equation}
is called the \textit{Plancherel formula}.
Bounded operators on $L^{2}(\rd)$ are denoted by $\mathcal{B}(\ltrn)$.

The standard convolution (for e.g. Schwartz functions)
\begin{equation*} 
f*_{0}g(x)=\int_{\rd} f(x-y)g(y) dy = \int_{\rd} f(y)g(x-y) dy
\end{equation*}
is commutative and it also makes sense for functions in the appropriate spaces $L^{p}(\rd)$ and $L^{q}(\rd)$.
The Fourier transform of a convolution is given by
$\widehat{f*_{0}g}=\widehat{f}\widehat{g}$.
We also define the convolution between a tempered distribution $A$ and a Schwartz function
\begin{equation*}
A*_{0}f(x) = \langle A, \tau_{x}\widetilde{f} \rangle, \qquad f \in \sco.
\end{equation*}

The following proposition is a mere modification of the theorem in Stein \cite{Ste} (VI.4.4., Proposition 2a). With minor modification its proof work in our case.
\begin{prop} \label{st0thm}
Let $(m_{t})_{t \in T}$, $T \subseteq \mathds{R}$ be a family of locally integrable and smooth functions on $\mathds{R}^{d} \backslash\{0\}$ and $M>d$. If $(m_{t})_{t \in T}$ satisfy
\begin{equation*} 
|\partial^{\alpha}_{\xi}m_{t}(\xi)|\leq c_{\alpha,M} c_{t} \|\xi\|^{-|\alpha|-M}, \ \ \ \alpha \in \mathds{N}^{d},
\end{equation*}
then the inverse Fourier transforms $K_{t}=m_{t}^{\vee}$ are smooth functions outside the origin and satisfy
\begin{equation*} 
|\partial^{\alpha}_{x}K_{t}(x)|\leq c'_{\alpha,M}c_{t}\|x\|^{-d-|\alpha|+M}, \alpha \in \mathds{N}^{d}.
\end{equation*}
\end{prop}

We say that $\psi$ is a \textit{bump function} if it is $[0,1]$-valued, we have $\psi(x)=1$ in some neighborhood of zero and $\psi(x)=0$ for $\|x\| \geq 1$ and moreover $\|\psi\|_{L^{1}(\rd)}=1$. Then
\begin{equation*}
\psi_{k}(x) = k^{d}\psi(kx), \qquad k \in \nd,
\end{equation*}
is an approximate identity of the Schwartz class.


\subsection{Heisenberg group and its dual space}
We consider the vector space $\mathfrak{h}= \rn \times \rn \times \mathds{R}$ as a Lie algebra with the commutator
\begin{equation*}
[(x_{1},x_{2},x_{3}), (y_{1},y_{2},y_{3})] = (0,0,(x_{1}\cdot y_{2}-x_{2}\cdot y_{1})),
\end{equation*}
and as a nonabelian Lie group ($\mathds{R}^{2n+1}$ is the underlying manifold) with the Campbell-Hausdorff multiplication
\begin{equation*}
xy=(x_{1},x_{2},x_{3})(y_{1},y_{2},y_{3})=(x_{1}+y_{1},x_{2}+y_{2},x_{3}+y_{3}+{1 \over 2}(x_{1}\cdot y_{2}-x_{2}\cdot y_{1})).
\end{equation*}
Such a group will be called the Heisenberg group.
Thus, in that case we change the notion from the section \ref{s1}.

The neutral element is the origin and the inverse of $x$ is $x^{-1}=-x$. We denote left and right translations by
\begin{equation*}
l_{x}(y)=xy, \qquad r_{x}(y)=yx, \qquad x,y \in \mathfrak{h}.
\end{equation*}
The algebra $\mathfrak{h}$ can be identified with the Lie algebra of left-invariant vextor fields $X_{1},\ldots,X_{2n+1}$ on the Heisenberg group.

The Heisenberg group is endowed with a family $(\delta_{s})_{s>0}$ of one-parameter automorphisms given by $\delta_{s}(x_{1},x_{2},x_{3})=(sx_{1},sx_{2},s^{2}x_{3})$ which are called dilations. Thus $\mathfrak{h}= \mathfrak{h}^{1} \oplus \mathfrak{h}^{2} \cong \mathds{R}^{2n} \times \mathds{R}$, where 
\begin{equation*}
\mathfrak{h}^{1}=\{x \in \mathfrak{h}: \delta_{s}(x)=sx\}, \qquad\ 
\mathfrak{h}^{2}=\{x \in \mathfrak{h}: \delta_{s}(x)=s^{2}x\}.
\end{equation*}
In depend of our convenience we will write elements of $\mathfrak{h}$ in coordinates as elements of $\mathds{R}^{2n+1}$ or $\mathds{R}^{2n} \times \mathds{R}$ or $\rn \times \rn \times \mathds{R}$. Similarly, for a multiindex $\alpha \in \mathds{N}^{2n+1}$ we will use notion for operators $\partial^{\alpha}$ or $T^{\alpha}$ depending on the context, for example
\begin{equation*}
\partial_{x}^{\alpha}f(x) = \partial_{x_{1}}^{\alpha_{1}} \partial_{x_{2}}^{\alpha_{2}} \partial_{x_{3}}^{\alpha_{3}}f(x_{1},x_{2},x_{3}), \qquad (x_{1},x_{2},x_{3}) \in \rn \times \rn \times \mathds{R}, \ \alpha_{1}, \alpha_{2} \in \nn, \alpha_{3} \in \mathds{N}.
\end{equation*}

We define the \textit{homogeneous norm} on the Heisenberg group
\begin{equation*}
|(u,t)|=(\|u\|^{2}+|t|)^{1 \over 2}, \qquad (u,t) \in \mathds{R}^{2n} \times \mathds{R}.
\end{equation*}
It satisfies $|x^{-1}|=|x|$ and $|\delta_{s}(x)|=s|x|$ for $x \in \mathfrak{h}$ i $s>0$.
By $l(\alpha)$ we also denote the \textit{homogeneous length} of a multiindex $\alpha \in \mathds{N}^{2n+1}$, i.e.
\begin{equation*}
l(\alpha)=(\alpha_{1}+\ldots+\alpha_{2n})+2\alpha_{2n+1}.
\end{equation*}
The dual space to the vector space $\mathfrak{h}$ is denoted by $\mathfrak{h}^{*} \cong \mathds{R}^{2n+1}$.
Then $\mathfrak{h}^{*}=(\mathfrak{h}^{1})^{*}\oplus (\mathfrak{h}^{2})^{*}$ and the space $(\mathfrak{h}^{1})^{*}$ will be sometimes idenified with the Phase Space $W = \mathds{R}^{2n}$.

The Schwartz class $\sco$ on the Heisenberg group is just the space $\mathcal{S}(\mathds{R}^{2n+1})$.
In the similar way we define other function spaces, in particular $\lth=L^{2}(\mathds{R}^{2n+1})$, and distributional spaces (if they are independent from the group structure; in other way it will be emphasized). 

The Lebesgue measure on the vector space $\mathfrak{h}$ is a (noninvariant) Haar measure on the group $\mathfrak{h}$.
As in the Abelian case we denote $\widetilde{f}(x)=f(x^{-1})$, because $x^{-1}=-x$.
The convolution on the Heisenberg group was defined in (\ref{splotc}).
The convolution of a tempered distribution $A$ and a Schwartz function $f$ is given by
\begin{equation*}
A*f(x)=\langle A, r_{x}\widetilde{f} \rangle, \qquad f*A(x)= \langle A, l_{x}\widetilde{f} \rangle, \qquad x \in \mathfrak{h}.
\end{equation*}
If $X$ is a left-invariant vector field on the group $\mathfrak{h}$, then
\begin{equation*}
X(f*g)(x)= \int_{\mathfrak{h}} f(y)X(g(y^{-1}x)) dy = \int_{\mathfrak{h}} f(y)(Xg)(y^{-1}x) dy = f*Xg(x).
\end{equation*}
The vector field $X_{2n+1}=\partial_{2n+1}$ is both left- and right-invariant. In particular,
\begin{equation*}
\partial_{2n+1}(f*g)=\partial_{2n+1}f*g=f*\partial_{2n+1}g.
\end{equation*}
There is an \textit{approximate identity} $(\psi_{k})_{k \in \mathds{N}}$ of the Schwartz class on the Heisenberg group, i.e.
$\psi_{k}*f \to f$ in the Schwartz class. 

We also consider the vector space $\mathds{R}^{2n} \times \mathds{R}$ with the commutators $[(u,t), (u',t')]_{\theta} = (0,\theta\{u,u'\})$ for $\theta \in [0,1)$ obtaining the family of Lie algebras $\mathfrak{h}_{\theta}$ and simultaneously Lie groups with the Campbell-Hausdorff multiplications 
\begin{equation*}
(u,t) \circ_{\theta} (u',t') = (u+u',t+t'+{1 \over 2}\theta\{u,u'\}).
\end{equation*}
The group $\mathfrak{h}_{0}$ is then the Abelian group $\mathds{R}^{2n+1}$ with the addition. On the other hand each group $\mathfrak{h}_{\theta}$, $\theta \in (0,1)$ is isomorphic with the Heisenberg group $\mathfrak{h}$.
Neutral element of each group is zero and the inverse of $(u,t)$ is $(-u,-t)$.
Let $*_{\theta}$ denote the convolution on the group $\mathfrak{h}_{\theta}$, $\theta \in (0,1)$.


Stone - von Neumann theorem gives a complete description of irreducible unitary representations on the Heisenberg group. Namely, each representation is equivalent to one of the following representations: representations $\chi_{\xi+i\eta}(x,y,t)=e^{i(x \cdot \xi + y \cdot \eta)}$ acting on $\mathds{C}$ or Schr{\"o}dinger representations $\pi^{\lambda}$, $\lambda \neq 0$ acting on the Hilbert space $\ltrn$.
We use the following notion
\begin{equation*}
\lambda^{1 \over 2}:=\sgn(\lambda)|\lambda|^{1 \over 2}=
\begin{cases}
\lambda^{1 \over 2} & \mbox{for } \lambda>0,\cr
-|\lambda|^{1 \over 2} & \mbox{for } \lambda<0.
\end{cases}.
\end{equation*}
For $x \in \mathds{R}^{n}$ we also denote $\lambda^{1 \over 2}x=(\lambda^{1 \over 2}x_{1},\ldots,\lambda^{1 \over 2}x_{n})$ (compare with a notion $\lambda^{1 \over 2}w$, $w \in W=\mathds{R}^{2n}$ on the Phase Space in the section \ref{s3}).

\textit{Schr{\"o}dinger representations} of the Heisenberg group indexed by the parameter $\lambda \in \mathds{R} \backslash \{0\}$ are given by
\begin{equation*}
\pi^{\lambda}_{(u,t)}f(x) = e^{i\lambda t}e^{i \lambda^{1 \over 2}x \cdot u_{2}} e^{{1 \over 2}i\lambda u_{1} \cdot u_{2}} f(x+|\lambda|^{1 \over 2}u_{1}), \qquad (u,t) \in \mathfrak{h}, f \in \ltrn.
\end{equation*}
The conjugate operator satisfies $(\pi^{\lambda}_{(u,t)})^{*}=\pi^{\lambda}_{(-u,-t)}$.

Schr{\"o}dinger representations $\pi^{\lambda}$ act on the Hilbert space $\mathcal{H}_{\pi^{\lambda}}=L^{2}(\mathds{R}^{n})$. We consider the space of \textit{smooth vectors} $\mathcal{H}^{\infty}_{\pi^{\lambda}}$, i.e. functions $f \in L^{2}(\mathds{R}^{n})$ such that (vector-valued) function $\psi_{f}^{\lambda}(u,t) = \pi^{\lambda}_{(u,t)}f$ is smooth.
It is known that $\mathcal{H}^{\infty}_{\pi^{\lambda}} = \mathcal{S}(\mathds{R}^{n})$ (see e.g. Corwin-Greenleaf \cite{CoGr} Cor. 4.1.2).

Sometimes it is useful to consider the space of smooth vectors in a weak sense $\mathcal{H}^{\infty, w}_{\pi^{\lambda}}$, i.e. functions $f \in L^{2}(\mathds{R}^{n})$ such that function $\psi_{f,g}^{\lambda}(u,t) = \langle (\pi^{\lambda}_{(u,t)})^{*}f, g \rangle_{\ltrn}$ which is called \textit{matrix coefficient} is smooth on $\mathfrak{h}$ for all $g \in L^{2}(\mathds{R}^{n})$. One can show (see Taylor \cite{Tay}, p.30-31) that 
$\mathcal{H}^{\infty}_{\pi^{\lambda}}=\mathcal{H}^{\infty,w}_{\pi^{\lambda}}$.

Another subclass of $\ltrn$ will be useful. We denote by $\mathcal{H}^{s}_{\pi^{\lambda}}$ the G{\"a}rding space, i.e.
\begin{equation*}
\mathcal{H}^{s}_{\pi^{\lambda}} = \spann\{\pi^{\lambda}_{\varphi}f: \varphi \in \mathcal{S}(\mathfrak{h}), f \in L^{2}(\mathds{R}^{n}) \}.
\end{equation*}
It has been shown in Dixmier-Malliavin \cite{DiMa} that the space of smooth vectors and the G{\"a}rding space are identical i.e. $\mathcal{H}^{s}_{\pi^{\lambda}}=\mathcal{H}^{\infty}_{\pi^{\lambda}}$.
We conclude from the above considerations the following theorem. 
\begin{cor}
On the Heisenberg group we have
\begin{equation*}
\mathcal{H}^{\infty}_{\pi^{\lambda}} = \mathcal{H}^{\infty,w}_{\pi^{\lambda}} = \mathcal{H}^{s}_{\pi^{\lambda}} = \mathcal{S}(\rn).
\end{equation*}
\end{cor}


\subsection{$L^{2}$-theory}
Theory of harmonic analysis on the Heisenberg group on the Phase Space one can find in Folland \cite{Fol} and Thangavelu \cite{Tha}.

Let $F \in \sco$. We consider the (operator-valued) \textit{group Fourier transform}
\begin{equation*}
\pi^{\lambda}_{F} = \int_{\mathfrak{h}} F(u,t) (\pi^{\lambda}_{(u,t)})^{*} du dt, \qquad \lambda \neq 0.
\end{equation*}
It is determinated by
\begin{equation*}
\langle \pi^{\lambda}_{F}f,g \rangle_{\ltrn} = \int_{\mathfrak{h}} \psi_{f,g}^{\lambda}(u,t) F(u,t) du dt, \ \ \ \lambda \neq 0, \ f,g \in \ltrn,
\end{equation*}
i.e. by using matrix coefficients.
For $F \in L^{1}(\mathfrak{h})$ the operator $\pi^{\lambda}_{F}$ is also well-defined and 
\begin{equation*}
|\langle \pi^{\lambda}_{F}f,g \rangle_{\ltrn}| \leq \|F\|_{L^{1}(\mathfrak{h})}\|f\|_{\ltrn}\|g\|_{\ltrn}.
\end{equation*}
Moreover, if $F,G \in L^{1}(\mathfrak{h})$, then $\pi^{\lambda}_{F}\pi^{\lambda}_{G}=\pi^{\lambda}_{F*G}$. 

If $F$ is a Schwartz function, then $\pi^{\lambda}_{F}$ is a Hilbert-Schmidt operator.
Indeed, we have
\begin{align*}
\pi^{\lambda}_{F}f(x) &= \int_{\mathfrak{h}} F(u,t) (\pi^{\lambda}_{(u,t)})^{*}f(x) du dt
\\
&= \int \int e^{-i\lambda t}e^{-i\lambda^{1 \over 2}(x-{1 \over 2}|\lambda|^{1 \over 2}u_{1})\cdot u_{2}}f(x-|\lambda|^{1 \over 2}u_{1}) F(u_{1},u_{2},t) du dt
\\
&= \int e^{-i\lambda^{1 \over 2}(x-{1 \over 2}|\lambda|^{1 \over 2}u_{1})\cdot u_{2}}f(x-|\lambda|^{1 \over 2}u_{1}) \mathcal{F}_{3}F(u_{1},u_{2},\lambda) du
\\
&= |\lambda|^{-{n \over 2}} \int \int e^{-i{1 \over 2}(x+y)\cdot \lambda^{1 \over 2}u_{2}}f(y) \mathcal{F}_{3}F(|\lambda|^{-{1 \over 2}}(x-y),u_{2},\lambda) dy du_{2}
\\
&= |\lambda|^{-{n \over 2}} \int f(y) \mathcal{F}_{2;3}F(|\lambda|^{-{1 \over 2}}(x-y),{1 \over 2}\lambda^{1 \over 2}(x+y),\lambda) dy = \int f(y) K^{\lambda}_{F}(x,y)  dy
\end{align*}
and thus $\pi^{\lambda}_{F}$ is such an operator with the kernel
\begin{equation*} \label{hspdo}
K^{\lambda}_{F}(x,y) = |\lambda|^{-{n \over 2}} \mathcal{F}_{2;3}F(|\lambda|^{-{1 \over 2}}(x-y),{1 \over 2}\lambda^{1 \over 2}(x+y),\lambda), \qquad x,y \in \rn.
\end{equation*}
Here, $\mathcal{F}$ with index denotes the partial Abelian Fourier transform on $\rn \times \rn \times \mathds{R}$.

The following formula is an analogue of the ordinary Plancherel formula on $\ltrn$.
\begin{thm}
Let $F \in \sco$. Then 
\begin{equation*}
\|F\|_{\lth}^{2} = (2\pi)^{-(n+1)}\int \|\pi^{\lambda}_{F}\|_{HS}^{2} |\lambda|^{n} d\lambda. 
\end{equation*}
\end{thm}
\begin{proof}
Let $K_{F}^{\lambda}$ be the kernel of the operator $\pi^{\lambda}_{F}$. Then
\begin{equation*}
\|\pi^{\lambda}_{F}\|^{2}_{HS} = \int \int |K_{F}^{\lambda}(x,y)|^{2} dx dy = |\lambda|^{-n} \int \int |\mathcal{F}_{2;3}F(|\lambda|^{-{1 \over 2}}(x-y),{1 \over 2}\lambda^{1 \over 2}(x+y),\lambda)|^{2} dx dy
\end{equation*}
Changing variables and using the ordinary Plancherel formula (\ref{plancz}) we get that the above is equal to
\begin{equation} \label{gwia}
(2\pi)^{-n}|\lambda|^{-n} \int \int |\mathcal{F}_{3}F(u_{1},u_{2},\lambda)|^{2} du_{1} du_{2} dv.
\end{equation}
Multiplying by $|\lambda|^{n}$, integrating with respect to $\lambda$ and using again the Plancherel formula we get the thesis.
\end{proof}

\begin{cor}
The map $F \mapsto \pi^{\lambda}_{F}$ defined for $F \in \sco$ can be extended to an unitary map from $\lth$ to the space of Hilbert-Schmidt operators.
\end{cor}
As a consequence, a convolution operator $f \mapsto A*f$ can be realized by using the group Fourier transform as $\pi^{\lambda}_{f} \to \pi^{\lambda}_{A}\pi^{\lambda}_{f}$. More precisely, one can define the operator
$\pi_{A}^{\lambda}$ on functions of the form $\pi_{f}^{\lambda}g$, where $f\in  L^{2}(\mathfrak{h})$, $g\in L^{2}(\rn)$, by 
\begin{equation*}
\pi_{A}^{\lambda}\pi_{f}^{\lambda}g=\pi_{A*f}^{\lambda}g. 
\end{equation*}
It is known (see e.g. Taylor \cite{Tay}) that the $\lth$-norm of the convolution operator $\Op(A)f=A*f$ can be computed as the supermum of the $\ltrn$-norms of the operators $\pi^{\lambda}_{A}$.
\begin{prop} \label{wpt}
Let $\Op(A)$ be a convolution operator bounded on $\lth$. Then, the operators $\pi^{\lambda}_{A}$ are uniformly bounded on $\ltrn$ and
\begin{equation} \label{suplam}
\|\Op(A)\|_{\lth \to \lth} = \sup_{\lambda \neq 0} \|\pi^{\lambda}_{A}\|_{\ltrn \to \ltrn}.
\end{equation}
\end{prop}


\subsection{$\sco$-convolvers}
We say that a tempered distribution $A$ is an \textit{$\sco$-convolver} (or $\mathcal{S}$-convolver on $\mathfrak{h}$), if 
\begin{equation*}
A*f \in \sco, \ \ \ f*A \in \sco,  \qquad f \in \sco.
\end{equation*}
One can also consider distributions which are only left or right $\sco$-convolvers. 
The basic information about $\sco$-convolvers one can find in Corwin \cite{Cor}.

Let $A$ be an $\sco$-convolver (i.e. both right and left). One can consider convolution operators with such a distribution. For the sake of attention let us suppose that we consider a convolution with $A$ on the left, i.e.
\begin{equation*}
\Op(A)f(x) = A*f(x) = \langle A, l_{x}\widetilde{f} \rangle, \qquad f \in \sco.
\end{equation*}

Such an operator can be extended to the operator $\Op(A):\scop \to \scop$ by the prescription
\begin{equation*}
\langle \Op(A)u, f \rangle = \langle u, \Op(A)^{*}f \rangle = \langle u, \Op(\widetilde{A})f \rangle, \qquad u \in \scop, f \in \sco.
\end{equation*}
Thus one can consider $\sco$-convolvers on every subspace of $\scop$.

The class of $\sco$-convolvers is denoted by $\mathscr{C}$. The class $\mathscr{C}$ contains tempered distributions with compact support and Schwartz functions.
For every $\sco$-convolver $A$ there is a sequence $A_{n}$ of Schwartz functions such that $A_{n}*f \to A*f$ in $\sco$ for all $f \in \sco$. It is enough to get $A_{n}=A*\theta_{n}$, where $\theta_{n}$ is an approximate identity in $\sco$.
\begin{prop}
The class $\mathscr{C}$ is closed under addition, convolution, differentiation and multiplication by polynomial. 
\end{prop}

In the similar way one can define $\mathcal{S}$-convolvers on (connected, simply connected, nilpotent) Lie groups. In the Abelian case $\mathcal{S}(\mathds{R}^{d})$-convolvers are characterized by the Abelian Fourier transform. Namely, $A$ is an $\mathcal{S}(\mathds{R}^{d})$-convolver iff its Fourier transform $\widehat{A}$ is a smooth function whose derivatives of any order are bounded by polynomials.

One can extend the definition of the group Fourier transform to all $\sco$-convolvers.
The following proposition one can find in Corwin \cite{Cor} (Proposition 2.3).
\begin{prop}
Let $A$ be an $\sco$-convolver. One can define (perhaps unbounded) an operator $\pi^{\lambda}_{A}$ whose domain contain  G{\"a}rding class $\mathcal{H}^{s}_{\pi^{\lambda}}$ (i.e the Schwartz class $\mathcal{S}(\rn)$). Moreover,
\begin{equation} \label{defpia}
\pi^{\lambda}_{A}\pi^{\lambda}_{\varphi}f=\pi^{\lambda}_{A \star \varphi}f, \qquad \varphi \in \sco, \ f \in \lth,
\end{equation}
$\pi^{\lambda}_{A}$ is continuous on $\mathcal{H}^{s}_{\pi^{\lambda}}$ and closable.
\end{prop}

Let us remark that in the case that $A$ is a distribution with compact support, the standard definition of the operator $\pi^{\lambda}_{A}$ is a little different. Namely, if $A \in \mathcal{E}'(\mathfrak{h})$, then the operator $\pi^{\lambda}_{A}:\mathcal{H}^{\infty}_{\pi^{\lambda}} \to \ltrn$ is define in a weak sense
\begin{equation} \label{defzwa}
\langle \pi^{\lambda}_{A}f,g \rangle_{\ltrn} = \langle A, \varphi_{f,g}^{\lambda} \rangle,
\end{equation}
where $\varphi_{f,g}^{\lambda}$ is a matrix coefficient.
Such an approach one can find in e.g Taylor \cite{Tay}.

The definition (\ref{defpia}) agrees with the definition (\ref{defzwa}), i.e. the operator $\pi^{\lambda}_{A}$ defined by (\ref{defzwa}) satisfies  $\pi^{\lambda}_{A*\psi}=\pi^{\lambda}_{A}\pi^{\lambda}_{\psi}$ for $\psi \in \sco$. Indeed, we have
\begin{align*}
\langle \pi^{\lambda}_{A*\psi}f, g \rangle_{\ltrn}
= \langle A*\psi, \varphi_{f,g}^{\lambda} \rangle
= \langle A, \varphi_{f,g}^{\lambda}*\widetilde{\psi} \rangle
= \langle A, \varphi_{\pi^{\lambda}_{\psi}f,g}^{\lambda} \rangle 
=\langle \pi^{\lambda}_{A}\pi^{\lambda}_{\psi}f, g \rangle_{\ltrn}. 
\end{align*}
We use the fact that $\varphi_{f,g}^{\lambda}*\widetilde{\psi} = \varphi_{\pi^{\lambda}_{\psi}f,g}^{\lambda}$. Indeed,
\begin{align*}
\varphi_{f,g}^{\lambda}*\widetilde{\psi}(h)
&= \int \varphi_{f,g}^{\lambda}(h(h')^{-1}) \widetilde{\psi}(h') dh'
= \int \psi(h') \int \pi^{\lambda}_{hh'}f(x) g(x) dx dh'
\\
&= \int \pi^{\lambda}_{h} \int \psi(h')\pi^{\lambda}_{h'}f(x) dh' g(x) dx
= \langle \pi^{\lambda}_{h}\pi^{\lambda}_{\psi}f, g \rangle_{\ltrn}
= \varphi_{\pi^{\lambda}_{\psi}f,g}^{\lambda}(h). 
\end{align*}


\section{Symbolic calculus} \label{s3}
\subsection{Symbolic calculus for pseudodifferential operators}
By $W$ we denote the Phase Space $\mathds{R}^{n} \times \mathds{R}^{n}$ with the symplectic form $\sigma(u,w)=u_{1} \cdot w_{2} - u_{2} \cdot w_{1}$, where $w=(w_{1}, w_{2}), u=(u_{1}, u_{2})$, $u_{i}, w_{i} \in \mathds{R}^{n}$, $i \in \{1,2\}$.

Let $\mathbf{g}$ be a Riemannian metric, i.e. a family $(\mathbf{g}_{w})_{w \in W}$ of positive quadratic forms on $W$, depending smoothly on $W$.
The conjugate metric $\mathbf{g}^{\sigma}$ of $\mathbf{g}$ is given by
\begin{equation*}
\mathbf{g}^{\sigma}_{w}(u)=\sup_{v \neq 0} {\sigma(u,v)^{2} \over \mathbf{g}_{w}(v)}.
\end{equation*}
We also define its gain factor by
\begin{equation*}
\tau_{\mathbf{g}}(w) = \inf_{u \neq 0} {\mathbf{g}^{\sigma}_{w}(u) \over \mathbf{g}_{w}(u)}.
\end{equation*}

It is said that a Riemannian metric is an \textit{admissible} metric (sometimes it is called H{\"o}rmander metric) if it satisfies the following properties:
\begin{enumerate}[\upshape (i)]
\item Slowness:
\begin{equation*}
\exists C, \ \ \ \ \ \mathbf{g}_{w}(u-w) \leq C^{-1} \Rightarrow \left({\mathbf{g}_{w}(v) \over \mathbf{g}_{u}(v)}\right)^{\pm 1} \leq C, \ \ \ \ \  u, w \in W, v \neq 0.
\end{equation*}
\item Temperance:
\begin{equation*}
\exists C, N, \ \ \ \ \ \left({\mathbf{g}_{w}(v) \over \mathbf{g}_{u}(v)}\right)^{\pm 1} \leq C(1+ \mathbf{g}^{\sigma}_{w}(u-w))^{N}, \ \ \ \ \ u, w \in W, v \neq 0.
\end{equation*}
\item Uncertainty principle: $ \ \ \ \tau_{\mathbf{g}}(w) \geq 1 \ \ \ \ \ $ $w \in W$. 
\end{enumerate}

The constants $C, N$ appearing above and any constants depending only on them are called the \textit{structural constants} of metric $\mathbf{g}$.

It is said that a positive function $m$ on $W$ is a \textit{$\mathbf{g}$-weight} if there are structural constants $C, N$ satisfying
\begin{equation*}
\exists C, \ \ \ \ \ \mathbf{g}_{w}(u-w) \leq C^{-1} \Rightarrow {m(w) \over m(u)} \leq C, \ \ \ \ \  u, w \in W.
\end{equation*}
and
\begin{equation*}
\exists C, N, \ \ \ \ \ \left({m(w) \over m(u)}\right)^{\pm 1} \leq C(1+ \mathbf{g}^{\sigma}_{w}(u-w)), \ \ \ \ \ u, w \in W.
\end{equation*}

In particular let us consider \textit{diagonal metrics}, i.e. metrics of the form
\begin{equation} \label{diag}
\mathbf{g}_{w}(u) = {\|u\|^{2} \over g(w)^{2}},
\end{equation}
where $g$ is a function on the Phase Space $W$.

Let $\mathbf{g}$ be a metric and let $m$ be a $\mathbf{g}$-weight. We say that a smooth function on $W$ is in the class $S(m,\mathbf{g})$ if the following seminorms
\begin{equation*}
\|a\|_{k;S(m,\mathbf{g})}=\sup_{w,\mathbf{g}_{X}(t_{j}) \leq 1} |a^{(k)}(w)(t_{1},\ldots,t_{k})| m(w)^{-1} \prod_{1 \leq j \leq k} \mathbf{g}_{w}(t_{j})^{-{1 \over 2}},
\end{equation*}
are finite for every $k \geq 0$.
Here $a^{(k)}$ is a $k$-multilinear form.

In the case of diagonal metrics (\ref{diag}) the class $S(m,\mathbf{g})$ is simply defined. Indeed $a \in S(m,\mathbf{g})$ iff
\begin{equation*}
|\partial^{\alpha}a(w)| \leq c_{\alpha}m(w)g(w)^{-|\alpha|}, \qquad \alpha \in \mathds{N}^{2n}.
\end{equation*}

Every $a \in \mathcal{S}(W)$ defines a linear map $A: \mathcal{S}(\mathds{R}^{n}) \to \mathcal{S}(\mathds{R}^{n})$ given by
\begin{equation*}
Af(x) =  \Op^{W}(a)f(x) = (2\pi)^{-n} \int \int e^{i(x-y)\cdot\xi} a\left({x+y \over 2},\xi\right) f(y) dy d\xi.
\end{equation*}
The function $a$ is called the \textit{symbol} of the pseudodifferential operator $A$. For $f,g \in \srn$ we define the \textit{Wigner transform} $\Psi_{f,g} \in \mathcal{S}(W)$ by
\begin{equation} \label{twig}
\Psi_{f,g}(x,\xi) = (2\pi)^{-n} \int e^{iy \cdot \xi} f(x+\frac{1}{2}y)\overline{g(x-\frac{1}{2}y)} dy.
\end{equation}
The weak version of a pseudodifferential operator
\begin{align*}
\langle Af, g \rangle_{\ltrn} &=
(2\pi)^{-n}\int \int \int e^{i(x-y)\cdot\xi} a\left({x+y \over 2},\xi\right) f(y) \overline{g(x)} dy d\xi dx
\\
= &(2\pi)^{-n}\int \int a(x,\xi) \int e^{iy \cdot \xi} f(x+\frac{1}{2}y)\overline{g(x-\frac{1}{2}y)} dy dx d\xi
= \langle a, \Psi_{f,g} \rangle,
\end{align*}
make sense for any $a \in \mathcal{S}'(W)$ and defines a linear operator which maps continuously $\mathcal{S}(\mathds{R}^{n})$ into $\mathcal{S}'(\mathds{R}^{n})$ (with better continuity properties in many cases). In particular, the identity operator satisfies  $I=\Op^{W}(1)$ which also follows from (\ref{foin}).


By $\Op^{W} S(m,\mathbf{g})$ we denote the class of operators whose symbols belong to $S(m,\mathbf{g})$. 
Let us also denote by $\#_{W}$ the composition of symbols which corresponds to the composition of operators, i.e.
\begin{equation*}
\Op^{W}(a \#_{W} b) = \Op^{W}(a)\Op^{W}(b).
\end{equation*}
The following properties are classical (see Lerner \cite{Ler}).
\begin{thm} \label{bigtm}
Suppose that $\mathbf{g}$ is an admissible metric and $m$, $m'$ are $\mathbf{g}$-weights.
\begin{enumerate}[\upshape (a)]
\item If $a \in S(m,\mathbf{g})$ and $b \in S(m',\mathbf{g})$, then for $k \in \mathds{N}$ there exist $k_{1}, k_{2} \in \mathds{N}$ such that
\begin{equation*}
\|a \#_{W} b\|_{k;S(mm',\mathbf{g})} \leq c_{k} \|a \|_{k_{1};S(m,\mathbf{g})} \|b\|_{k_{2};S(m',\mathbf{g})}.
\end{equation*}
In particular $S(m,\mathbf{g}) \#_{W} S(m',\mathbf{g}) \subseteq S(mm',\mathbf{g})$.
\item Operators in $\Op^{W} S(m,\mathbf{g})$ are bounded on $L^{2}(\mathds{R}^{n})$ for $m=1$. Moreover, there exist $k \in \mathds{N}$, $C>0$ such that
\begin{equation*}
\|\Op^{W}(a)\|_{L^{2} \to L^{2}} \leq C\|a \|_{k;S(1,\mathbf{g})}.
\end{equation*}
\end{enumerate}
\end{thm}

Let us recall the celeberated Beals \cite{Be2} characterization.
\begin{thm} \label{bealsthm}
The symbol $a=a(x,\xi)$ of a pseudodifferential operator $A=\Op^{W}(a)$ belongs to the class $S(1,\mathbf{1})$ if and only if $A$ and its iterated commutators with derivatives $\partial_{\xi_{i}}$ and multiplications by $x_{j}$ are bounded on $\ltrn$. Moreover, for every $k \in \mathds{N}$, there exists $j \in \mathds{N}$ such that
\begin{equation*}
\|a\|_{k;S(1,\mathbf{1})} \leq c_{k} \max_{|\alpha| \leq j}\|\Op^{W}(\partial^{\alpha}a)\|_{\ltrn \to \ltrn}.
\end{equation*}
\end{thm}


\subsection{Weights and weight functions}
We say that a function $g$ is a \textit{weight function on $\mathfrak{h}^{*}$}, if it satisfies
\begin{enumerate}[\upshape (i)]
\item $\exists C>0 \qquad \|u-w\| \leq C^{-1}g(w,\lambda) \Rightarrow \left({g(u,\lambda) \over g(w,\lambda)}\right)^{\pm 1} \leq C, \qquad  u, w \in W, \ \lambda \neq 0,$
\item $\exists C>0, N \in \mathds{N} \qquad \left({g(u,\lambda) \over g(w,\lambda)}\right)^{\pm 1} \leq C(1+ |\lambda|^{-1} g(w,\lambda)\|u-w\|)^{N}, \qquad u, w \in W, \ \lambda \neq 0$,
\item $ g(w,\lambda)^{2} \geq |\lambda|, \qquad w \in W$. 
\end{enumerate}
The above conditions are called slowness, temperance and uncertainty principle.
We say that a function $m$ is a \textit{weight for a weight function} $g$ on $\mathfrak{h}^{*}$, if
\begin{enumerate}[\upshape (i)]
\item $\exists C>0 \qquad  \|u-w\| \leq C^{-1}g(w,\lambda) \Rightarrow \left({m(w,\lambda) \over m(u,\lambda)}\right)^{\pm 1} \leq C, \qquad  u, w \in W, \lambda \neq 0,$
\item $\exists C>0, N \in \mathds{N} \qquad \left({m(w,\lambda) \over m(u,\lambda)}\right)^{\pm 1} \leq C(1+ |\lambda|^{-1} g(w,\lambda)\|u-w\|)^{N}, \qquad u, w \in W, \lambda \neq 0$.
\end{enumerate}
Such conditions will be also called slowness and temperance (with respect to weight function $g$).

The following corollary follows directly from the definitions of weight function and weight.
\begin{rem}
A function $g$ is a weight function on $\mathfrak{h}^{*}$ iff metrics $\mathbf{g}_{\lambda}$ on $W$ indexed by the parameter $\lambda \in \mathds{R} \backslash \{0\}$ given by
\begin{equation} \label{gmet}
(\mathbf{g}_{\lambda})_{w}(u)={|\lambda|\|u\|^{2} \over g_{(\lambda)}(w)^{2}},
\end{equation}
are H{\"o}rmander metrics and structural constants (of metrics $\mathbf{g}_{\lambda}$) are independent from  $\lambda$.
Moreover, a function $m$ is a weight for $g$ on $\mathfrak{h}^{*}$ iff the functions $m_{(\lambda)}$ are weights for the metrics $\mathbf{g}_{\lambda}$ on $W$ (then the structural constants of weights are also independent from $\lambda$).
\end{rem}

\begin{exa}
The function
\begin{equation*}
\rho(w,\lambda) = (1+\|w\|^{2}+|\lambda|^{2})^{{1 \over 2}}.
\end{equation*}
is a weight function on $\mathfrak{h}^{*}$.
\end{exa}
\begin{proof}
The uncertainty principle 
\begin{equation*}
|\lambda| \leq 1+\|w\|^{2}+|\lambda|^{2}
\end{equation*}
is satisfied.

We show that there is a constant $C>0$ such that
\begin{equation} \label{duzece}
\|w-u\|^{2}  \leq C^{-2}(1+\|w\|^{2}+|\lambda|^{2}) \qquad \Rightarrow \qquad
\left({1+\|u\|^{2}+|\lambda|^{2} \over 1+\|w\|^{2}+|\lambda|^{2}}\right)^{\pm 1} \leq C^{2},
\end{equation}
i.e. the slowness condition is satisfied.
If $\|u\| \leq \|w\|$, then $1+\|u\|^{2}+|\lambda|^{2} \leq 1+\|w\|^{2}+|\lambda|^{2}$ and thus $C \geq 1$. On the other side,
\begin{align*}
1+\|w\|^{2}+|\lambda|^{2} &\leq 2(1+\|u\|^{2}+\|u-w\|^{2}+|\lambda|^{2}) 
\\
&\leq  2(1+\|u\|^{2}+|\lambda|^{2})+ 2C^{-2}(1+\|w\|^{2}+|\lambda|^{2})
\end{align*}
and then $1+\|w\|^{2}+|\lambda|^{2} \leq {2 \over 1-2C^{-2}}(1+\|u\|^{2}+|\lambda|^{2})$. It is enough to $C \geq \sqrt{2}$. The case $\|w\| \leq \|u\|$ is similar.

We show that $\rho$ is tempered, i.e.
\begin{equation} \label{perte}
\left({1+\|u\|^{2}+|\lambda|^{2} \over 1+\|w\|^{2}+|\lambda|^{2}}\right)^{\pm 1} \leq C^{2}\left(1+|\lambda|^{-1}(1+\|w\|^{2}+|\lambda|^{2})\|w-u\|^{2}  \right).
\end{equation}
As $\|u\|^{2} \leq 2(\|w\|^{2}+\|u-w\|^{2})$, we get
\begin{equation*}
{1+\|u\|^{2}+|\lambda|^{2} \over 1+\|w\|^{2}+|\lambda|^{2}} 
\leq 2{\|u-w\|^{2}+1+\|w\|^{2}+|\lambda|^{2} \over 1+\|w\|^{2}+|\lambda|^{2}}
\leq 2\left(1+{\|u-w\|^{2} \over 1+\|w\|^{2}+|\lambda|^{2}}\right) 
\end{equation*}
and thesis follows from the uncertainty principle.

The second case is similar.
\end{proof}

The similar example (formulated in terms of metrics) one can find in Fischer-Ruzhansky \cite{FiRu} and Bahouri-Fermian-Kammerer-Gallagher \cite{BKG}. They considered the metrics ${|\lambda|\|u\|^{2} \over 1+|\lambda|(1+\|w\|^{2})}$.

\begin{exa} \label{p54}
The function
\begin{equation*}
l(w,\lambda) = 1+\log(1+\|w\|^{2}+|\lambda|^{2}).
\end{equation*}
is a weight for the weight function $\rho$.
\end{exa}
\begin{proof}
We assume that
\begin{equation*}
\|u-w\|^{2} \leq C^{-2} (1+\|w\|^{2}+|\lambda|^{2})
\end{equation*}
for some (large enough) $C$ appearing in (\ref{duzece}). Then
\begin{equation*}
1+\|u\|^{2}+|\lambda|^{2} \leq 1+2(\|w\|^{2}+\|u-w\|^{2})+|\lambda|^{2} \leq (2+2C^{-2})(1+\|w\|^{2}+|\lambda|^{2})
\end{equation*}
and we get slowness
\begin{align*}
{1+\log(1+\|u\|^{2}+|\lambda|^{2}) \over 1+\log(1+\|w\|^{2}+|\lambda|^{2})}
=1+{\log\left({1+\|u\|^{2}+|\lambda|^{2} \over 1+\|w\|^{2}+|\lambda|^{2}}\right) \over 1+\log(1+\|u\|^{2}+|\lambda|^{2})} \leq 1+\log(2+2C^{-2}) \leq C^{2}.
\end{align*}
The second case is similar.

One the other hand, using (\ref{perte}) we get
\begin{align*}
{1+\log(1+\|u\|^{2}+|\lambda|^{2}) \over 1+\log(1+\|w\|^{2}+|\lambda|^{2})}
&=1+{\log\left({1+\|u\|^{2}+|\lambda|^{2} \over 1+\|w\|^{2}+|\lambda|^{2}}\right) \over 1+\log(1+\|u\|^{2}+|\lambda|^{2})}
\\
&\leq 1+\log \left(C^{2}(|\lambda|^{-1}(1+\|w\|^{2}+|\lambda|^{2})\|u-w\|^{2})\right)
\\
&\leq
C^{2}\left(1+|\lambda|^{-1}(1+\|w\|^{2}+|\lambda|^{2})\|u-w\|^{2}\right).
\end{align*}
\end{proof}
It is not hard to see that $g^{M}$ is a weight for the weight function $g$ on $\mathfrak{h}^{*}$ for $M \in \mathds{R}$.


We say that a tempered distribution $A$ is in the class $\smgd$ if its Abelian Fourier transform $\widehat{A}$ is in the class $\smgh$ of smooth functions $a$ on $\mathfrak{h}^{*}$ for which the following seminorms
\begin{equation*}
\|a\|_{k;\smgh} = \max_{|\alpha|+\beta \leq k} \sup_{(w,\lambda) \in \mathfrak{h}^{*}} |\partial_{w}^{\alpha}\partial^{\beta}_{\lambda} a(w,\lambda)| m(w,\lambda)^{-1} g(w,\lambda)^{|\alpha|+\beta}. 
\end{equation*}
are finite. We also use the following distributional seminorms
\begin{equation*}
\|A\|_{k;\smgd}=\|\widehat{A}\|_{k;\smgh}
\end{equation*}
which define the class $\smgd$.


The classes $S(m,g;\mathfrak{h}^{*})$ can be characterized by using "quantizied" symbols. This will be crucial for the symbolic calculus on the Heisenberg group.
Hereunder we write $\mathfrak{h}=\mathds{R}^{2n} \times \mathds{R}$. In particular, for $\alpha \in \mathds{N}^{2n+1}$ we have $\partial^{\alpha}=\partial_{1}^{\alpha_{1}}\partial_{2}^{\alpha_{2}}$, $\alpha_{1} \in \mathds{N}^{2n}$, $\alpha_{2} \in \mathds{N}$.
\begin{prop} \label{chse}
For all $k \in \mathds{N}$ we have
\begin{equation*}
\|a\|_{k;S(m,g;\mathfrak{h}^{*})} = \max_{k_{1}+k_{2} \leq k} \sup_{\lambda \neq 0} \|(\partial^{k_{2}}_{2}a)_{(\lambda)}\|_{k_{1};S(m_{(\lambda)}g_{(\lambda)}^{-k_{2}},\mathbf{g}_{\lambda})}.
\end{equation*}
In other words, a smooth function $a$ is in the class $S(m,g;\mathfrak{h}^{*})$ iff for all $l \in \mathds{N}$, the functions $(\partial^{l}_{2}a)_{(\lambda)}$ are in the classes $S(m_{(\lambda)}g_{(\lambda)}^{-l},\mathbf{g}_{\lambda})$ with the seminorms independent from $\lambda$.
\end{prop}
\begin{proof}
Let us notice, what follows from the definition of seminorms $\|\cdot\|_{k;S(m,g;\mathfrak{h}^{*})}$ that
\begin{equation*}
\|a\|_{k;S(m,g;\mathfrak{h}^{*})}=\max_{k_{1}+k_{2} \leq k} \|\partial_{2}^{k_{2}}a\|_{k_{1};S_{0}(mg^{-k_{2}},g;\mathfrak{h}^{*})}, \ \ \ k \in \mathds{N}.
\end{equation*}
Thus, it is enough to show
\begin{equation*}
\|a\|_{l;S_{0}(m,g;\mathfrak{h}^{*})}= \sup_{\lambda \neq 0} \|a_{(\lambda)}\|_{l;S(m_{(\lambda)},\mathbf{g}_{\lambda})},
\end{equation*}
which follows from the form of (\ref{gmet}) and the following equity $|\partial^{\alpha}_{w}a_{(\lambda)}|=|\lambda|^{|\alpha| \over 2}|(\partial_{1}^{\alpha}a)_{(\lambda)}|$.
\end{proof}


\begin{fac} \label{f54}
If $A \in S(m,g;\mathfrak{h})$, then $A$ is an $\sco$-convolver.
\end{fac}
\begin{proof}
We show that $A$ can be decompose as a sum of a distribution with compact support and a Schwartz function. Let $\alpha \in \mathds{N}^{2n+1}$. For some $K=K(\alpha)$ we have then
\begin{equation*}
\Delta^{N}T^{\alpha}\widehat{A} \in \lth, \qquad N \geq K
\end{equation*}
and thus
\begin{equation*}
\|\cdot\|^{2N}\partial^{\alpha}A \in \lth, \qquad N \geq K.
\end{equation*}
In particular,
\begin{equation*}
|\partial^{\alpha}A(x)| \leq c_{\alpha,k} \|x\|^{-k}, \ \ \ \|x\| \geq 1, k \in \mathds{N}
\end{equation*}
and then $(1-\varphi)A$ is a Schwartz function, where $\varphi$ is a bump function.
\end{proof}


\subsection{Further properties of $\sco$-convolvers}
The class of $\sco$-convolvers is closed under convolution and this allows us to define an operation 
\begin{equation*}
a \# b = (a^{\vee}*b^{\vee})^{\wedge}
\end{equation*} 
which is called twisted product for such $a, b \in \scop$ that $a^{\vee}$, $b^{\vee}$ are $\sco$-convolvers.

Let $w \in W \cong (\mathfrak{h}^{1})^{*}$. We denote $\lambda^{1 \over 2}w=(|\lambda|^{1 \over 2}w_{2},\lambda^{1 \over 2}w_{1})$. Recall that $\lambda^{1 \over 2}x$ for $x \in \rn$ was defined before (and consequently we have different definitions $\lambda^{1 \over 2}w$ for $w \in \mathds{R}^{2n}$ and $\lambda^{1 \over 2}x$ for $x \in \rn$).
For a function $a$ on $\mathfrak{h}^{*}$, let $a_{(\lambda)}$ be a function on $(\mathfrak{h}^{1})^{*}$ given by $a_{(\lambda)}(w)=a(\lambda^{1 \over 2}w,\lambda)$.

\begin{prop}
$\pi^{\lambda}_{F}$ is a pseudodifferential operator of Weyl-H{\"o}rmander
\begin{equation*}
\pi^{\lambda}_{F} = \Op^{W}(a_{\lambda})
\end{equation*}
with the symbol $a_{\lambda} = \widehat{F}_{(\lambda)}$ depending on the parameter $\lambda \in \mathds{R} \backslash \{0\}$.
\end{prop} 
\begin{proof}
We checked in (\ref{hspdo}) that
\begin{equation*}
\pi^{\lambda}_{F}f(x) = \int K_{F}^{\lambda}(x,y) f(y) dy
= (2\pi)^{n+1 \over 2}|\lambda|^{-{n \over 2}} \int f(x) \mathcal{F}_{2;3}F(|\lambda|^{-{1 \over 2}}(x-y),{1 \over 2}\lambda^{1 \over 2}(x+y),\lambda) dy.
\end{equation*}
Using the inverse Abelian Fourier transform we get that the above is equal to
\begin{equation*}
(2\pi)^{-n} \int \int e^{i(x-y) \cdot \xi} f(y) \widehat{F}(|\lambda|^{1 \over 2}\xi,{1 \over 2}\lambda^{1 \over 2}(x+y),t) dy d\xi
= \Op^{W}(\widehat{F}_{(\lambda)})f(x).
\end{equation*}
\end{proof}

\begin{lem} \label{szele}
Let $A$ be an $\sco$-convolver. Then,
$\pi^{\lambda}_{A}=\Op^{W}(\widehat{A}_{(\lambda)})$.
\end{lem}
\begin{proof}
If $A$ is an $\sco$-convolver, then, in particular, it is a tempered distribution. Putting
\begin{equation*}
\langle \widehat{A}_{(\lambda)}, f \rangle = \langle \widehat{A}((\cdot),\lambda), |\lambda|^{-{1 \over 2}}f(\lambda^{1 \over 2}(\cdot)) \rangle, \qquad \lambda \neq 0, \ f \in \mathcal{S}(W),
\end{equation*}
we get $\widehat{A}_{(\lambda)} \in \mathcal{S}'(W)$. If $\theta_{k}$ is an approximate identity, then  $(A*\theta_{k})^{\wedge}_{(\lambda)} \to \widehat{A}_{(\lambda)}$ in $\mathcal{S}'(W)$.

Thus the pseudodifferential operator $\Op^{W}(\widehat{A}_{(\lambda)})$ can be defined by using Wigner transformation $\Psi_{f,g}$ and we have covengence
\begin{align*}
\langle \Op^{W}(\widehat{A}_{(\lambda)})f,g \rangle_{\ltrn} &= \langle \widehat{A}_{(\lambda)}, \Psi_{f,g} \rangle = \lim \langle (A*\theta_{k})^{\wedge}_{(\lambda)}, \Psi_{f,g} \rangle 
\\
&=\lim \langle \Op^{W}((A*\theta_{k})^{\wedge}_{(\lambda)})f,g \rangle_{\ltrn},  \qquad f,g \in \sco.
\end{align*}
On the other hand, $A*\theta_{k} \in \sco$ and we have strong convergence
\begin{equation*}
\lim \Op^{W}((A*\theta_{k})^{\wedge}_{(\lambda)}) = \lim \pi^{\lambda}_{A*\theta_{k}}
= \pi^{\lambda}_{A}.
\end{equation*}
The above equities give the thesis.
\end{proof}

By the above lemma and Fact \ref{f54},
\begin{equation*}
\pi^{\lambda}_{A} = \Op^{W}(\widehat{A}_{(\lambda)}), \qquad A \in \smgd.
\end{equation*}
Moreover, $(A*\theta_{n})^{\wedge}_{(\lambda)} \in S(m_{(\lambda)},\mathbf{g}_{\lambda})$ uniformly and we have the strong convergence $\pi^{\lambda}_{A*\theta_{n}} \to \pi^{\lambda}_{A}$.

It follows directly form the definition that
\begin{equation*}
\pi^{\lambda}_{A}\pi^{\lambda}_{B}=\pi^{\lambda}_{A*B}, \qquad \lambda \neq 0,
\end{equation*}
for $\sco$-convolvers $A, B$. Comparing the symbols of the above pseudodifferential operators we get the following formula
\begin{equation} \label{lorbi}
\widehat{A}_{(\lambda)} \#_{W} \widehat{B}_{(\lambda)} = (\widehat{A} \# \widehat{B})_{(\lambda)}.
\end{equation}
Notice that $\pi^{\lambda}_{\delta_{0}}f=f$, i.e. "representation" of Dirac delta is the identity operator.

Let us use the following notion
\begin{align*}
[T_{1},T_{2}]_{*}(f,g)=\sum_{j=1}^{n} (T_{1,j}f*T_{2,j}g-T_{2,j}f*T_{1,j}g),
\\
[\partial_{1},\partial_{2}]_{\#}(a,b)=\sum_{j=1}^{n} (\partial_{1,j}a \# \partial_{2,j}b-T_{2,j}a \# \partial_{1,j}b),
\end{align*}
changing alternatively the index of convolution or twisted product. In particular, the lack of notion of twisted product refers to the ordinary product, i.e.
\begin{equation*}
[\partial_{1},\partial_{2}](a,b)(x,\xi)= \sum_{j=1}^{n} (\partial_{1,j}a\partial_{2,{j}}b-\partial_{2,j}a \partial_{1,j}b).
\end{equation*}

\begin{prop} \label{proz}
Let $A, B \in \mathscr{C}$ and let one of them be the sum of a distribution with compact support and a Schwartz function. Then, for every $N>0$ we have the following decomposition in $\mathscr{C}$
\begin{equation} \label{dcmp}
A*B=
\sum_{0 \leq j<N} {1 \over j!}(-{1 \over 2}\partial_{3}[T_{1},T_{2}]_{*_{0}})^{j}(A,B) + R_{N}(A,B),
\end{equation}
where
\begin{equation*}
R_{N}(A,B) = \int_{0}^{1} (1-\theta)^{N-1} {1 \over (N-1)!}(-{1 \over 2}\partial_{3}[T_{1},T_{2}]_{*_{\theta}})^{N}(A,B) \ d\theta.
\end{equation*}
Let $a=\widehat{A}$, $b=\widehat{B}$. Equivalently, by the Abelian Fourier transform we get
\begin{equation*} 
a \# b=
\sum_{0 \leq j<N} {1 \over j!}({i \over 2}T_{3}[\partial_{1},\partial_{2}])^{j}(a,b) + R_{N}(a,b),
\end{equation*}
where
\begin{equation} \label{reha}
R_{N}(a,b) = \int_{0}^{1} (1-\theta)^{N-1} {1 \over (N-1)!}({i \over 2}T_{3}[\partial_{1},\partial_{2}]_{\#_{\theta}})^{N}(a,b) \ d\theta.
\end{equation}
\end{prop}
\begin{proof}
\textit(sketch) At first let $A$, $B$ be Schwartz functions. If we write the Taylor series for $A$ we get the thesis. Similarly is in the case when $A$ is a Schwartz function, and $B$ is an $\sco$-convolver. In general case we show equity for $N=1$, other cases are similar. Let $A, B \in \mathscr{C}$. We get
\begin{equation*}
\langle A*B, f \rangle = \langle A, f*\widetilde{B} \rangle = \langle A, f*_{0}\widetilde{B} -{1 \over 2} \sum_{j=1}^{n} \int_{0}^{1} \partial_{3}(T_{1,j}f*_{\theta}T_{2,j}\widetilde{B}-T_{2,j}f*_{\theta}T_{1,j}\widetilde{B}) d\theta \rangle
\end{equation*}
Using Leibniz's rule we get that the above is equal to
\begin{align*}
\langle A*_{0}B, f\rangle +{1 \over 2} \sum_{j=1}^{n} \int_{0}^{1} \langle \partial_{3}A, (T_{1,j}(f*_{\theta}T_{2,j}\widetilde{B})-T_{2,j}(f*_{\theta}T_{1,j}\widetilde{B})) \rangle d\theta
\\
= \langle A*_{0}B -{1 \over 2} \sum_{j=1}^{n} \int_{0}^{1} \partial_{3}(T_{1,j}A*_{\theta}T_{2,j}B - T_{2,j}A*_{\theta}T_{1,j}B) d\theta, f \rangle.
\end{align*}
The changing of integral is justified, because one of distributions is the sum of a distribution with compact support and a Schwartz function.
\end{proof}

For $j<N$ let $S_{j}(a,b)$ denotes $j$-th element in the sum appearing in (\ref{dcmp}), i.e.
\begin{equation*}
S_{j}(a,b) = {1 \over j!}({i \over 2}T_{3}[\partial_{1},\partial_{2}])^{j}(a,b).
\end{equation*}
It is easy to see that $S_{0}(a,b)=ab$. We will also use an explicit form of $S_{1}(a,b)$
\begin{equation*}
S_{1}(a,b)(w,\lambda) = {i\lambda \over 2} \sum_{j=1}^{n} (\partial_{1,j}a\partial_{2,j}b-\partial_{2,j}a\partial_{1,j}b)(w,\lambda).
\end{equation*}


The following formula from G{\l}owacki \cite{Glo6} is an analogue of the Leibniz rule.
\begin{prop}
Let $A, B \in \mathscr{C}$. Then
\begin{equation*}
T^{\alpha}(A*B)= \sum_{l(\beta)+l(\gamma)=l(\alpha)} c_{\beta,\gamma} T^{\beta}A*T^{\gamma}B.
\end{equation*}
Let $a=\widehat{A}$ and $b=\widehat{B}$. Equivalently, by the Fourier transform we get
\begin{equation} \label{lzs}
(i\partial^{\alpha})(a\#b)
=\sum_{l(\beta)+l(\gamma)=l(\alpha)} c_{\beta,\gamma} (i\partial^{\beta})a \# (i\partial^{\gamma})b.
\end{equation}
\end{prop}
The above formula is true for any induced convolution on homogeneous Lie groups, in particular for the convolution $*_{\theta}$, $\theta \in (0,1)$. As one can see in the author's work \cite{KaBe} the constants $c_{\beta,\gamma}(\theta)$ are uniformly bounded with respect to $\theta \in (0,1)$.


\subsection{Composition theorem}
The twisted product $a \# b = (a^{\vee}*b^{\vee})^{\wedge}$ defined for tempered distributions $a$, $b$, such that $a^{\vee}, b^{\vee}$ are $\sco$-convolvers has better continuity properties if $a$, $b$ are functions from some classes of symbols. 

\begin{thm} \label{compthm}
Suppose that $g$ is a weight function on $\mathfrak{h}^{*}$ and $m$, $m'$ are $g$-weights.
If $a \in \smgh$ and $b \in \smghp$, then for all $k \in \mathds{N}$ there exist $k_{1}, k_{2} \in \mathds{N}$ such that
\begin{equation*}
\|a \# b\|_{k;\smghpp} \leq c_{k} \|a \|_{k_{1};\smgh} \|b\|_{k_{2};\smghp}.
 \end{equation*}
In particular, $\smgh \# \smghp \subseteq \smghpp$.
\end{thm}
\begin{proof}
Let $a=\widehat{A} \in S(m,g;\mathfrak{h}^{*}), b=\widehat{B} \in S(m',g;\mathfrak{h}^{*})$ and for $\lambda \neq 0$ let $\mathbf{g}_{\lambda}$ be diagonal metrics (\ref{gmet}). Then $a_{(\lambda)} \in S(m_{(\lambda)},\mathbf{g}_{\lambda})$, $b_{(\lambda)} \in S(m'_{(\lambda)},\mathbf{g}_{\lambda})$ with seminorms independent from $\lambda$.
By Leibniz's rule (\ref{lzs}) and (\ref{lorbi}), (i.e. $(a\#b)_{(\lambda)}=a_{(\lambda)} \#_{W} b_{(\lambda)}$) we get that for all $l \in \mathds{N}$,
\begin{align*}
(\partial^{l}_{2}(a \# b))_{(\lambda)} 
&=\sum_{l(\beta)+l(\gamma)= 2l} c_{\beta,\gamma} (\partial^{\beta}a \# \partial^{\gamma}b)_{(\lambda)}
=\sum_{l(\beta)+l(\gamma)= 2l} c_{\beta,\gamma} (\partial^{\beta}a)_{(\lambda)} \#_{W} (\partial^{\gamma}b)_{(\lambda)}
\\
&=\sum_{l(\beta)+l(\gamma)= 2l} c_{\beta,\gamma} (\lambda^{1 \over 2})^{-|\beta_{1}|-|\gamma_{1}|} \partial^{\beta_{1}}(\partial_{2}^{\beta_{2}}a)_{(\lambda)} \#_{W} \partial^{\gamma_{1}}(\partial_{2}^{\gamma_{2}}b)_{(\lambda)}
\end{align*}
Using the equivalent seminorms form Proposition \ref{chse} we obtain
\begin{multline*}
\|a \# b\|_{k;S(mm',g;\mathfrak{h}^{*})} = \max_{k_{1}+k_{2} \leq k} \sup_{\lambda \neq 0} \|(\partial^{k_{2}}_{2}(a \#b))_{(\lambda)}\|_{k_{1}; S(m_{(\lambda)}m'_{(\lambda)}g_{(\lambda)}^{-k_{2}},\mathbf{g}_{\lambda})}
\\
\leq
\max_{k_{1}+k_{2} \leq k} \sup_{\lambda \neq 0} \sum_{l(\beta)+l(\gamma)= 2k_{2}} |c_{\beta,\gamma}| |\lambda|^{-{|\beta_{1}|+|\gamma_{1}| \over 2}}
\|\partial^{\beta_{1}}(\partial_{2}^{\beta_{2}}a)_{(\lambda)} \#_{W} \partial^{\gamma_{1}}(\partial_{2}^{\gamma_{2}}b)_{(\lambda)}\|_{k_{1}; S(m_{(\lambda)}m'_{(\lambda)}g_{(\lambda)}^{-k_{2}},\mathbf{g}_{\lambda})}.
\end{multline*}
Let us notice that
\begin{equation*}
g_{(\lambda)(w)}^{-k_{2}} \geq |\lambda|^{-{|\beta_{1}|+|\gamma_{1}| \over 2}} \tau_{\mathbf{g}_{\lambda}}^{-{|\beta_{1}|+|\gamma_{1} \over 2}|} g_{(\lambda)(w)}^{-|\beta_{2}|-|\gamma_{2}|},
\end{equation*}
(recall that $\tau_{\mathbf{g}_{\lambda}}(w)=|\lambda|^{-1}g_{(\lambda)}(w)^{2}$) and thus the above can be estimated by
\begin{equation*}
\max_{k_{1}+k_{2} \leq k} \sup_{\lambda \neq 0} c_{k_{2}} \|\partial^{\beta_{1}}(\partial_{2}^{\beta_{2}}a)_{(\lambda)} \#_{W} \partial^{\gamma_{1}}(\partial_{2}^{\gamma_{2}}b)_{(\lambda)}\|_{k_{1}; S(m_{(\lambda)}m'_{(\lambda)}\tau_{\mathbf{g}_{\lambda}}^{-{1 \over 2}(|\beta_{1}|+|\gamma_{2}|)} g_{(\lambda)}^{-\beta_{2}-\gamma_{2}},\mathbf{g}_{\lambda})}.
\end{equation*}
By the symbolic calculus for pseudodifferential operators on the Phase Space, there exist $k_{1,1}, k_{1,2}$ such that the above is bounded by
\begin{multline*}
\max_{k_{1}+k_{2} \leq k} \sup_{\lambda \neq 0} c_{k_{2}} \|\partial^{\beta_{1}}(\partial_{2}^{\beta_{2}}a)_{(\lambda)}\|_{k_{1,1}; S(m_{(\lambda)}\tau_{\mathbf{g}_{\lambda}}^{-{1 \over 2}|\beta_{1}|} g_{(\lambda)}^{-\beta_{2}},\mathbf{g}_{\lambda})} \|\partial^{\gamma_{1}}(\partial_{2}^{\gamma_{2}}b)_{(\lambda)}\|_{k_{1,2}; S(m'_{(\lambda)}\tau_{\mathbf{g}_{\lambda}}^{-{1 \over 2}|\gamma_{1}|} g_{(\lambda)}^{-\gamma_{2}},\mathbf{g}_{\lambda})}
\\
\leq
\max_{k_{1}+k_{2} \leq k} \sup_{\lambda \neq 0} c_{k_{2}} \|(\partial_{2}^{\beta_{2}}a)_{(\lambda)}\|_{k_{1,1};
S(m_{(\lambda)} g_{(\lambda)}^{-\beta_{2}},\mathbf{g}_{\lambda})}
\|(\partial_{2}^{\gamma_{2}}b)_{(\lambda)}\|_{k_{1,2}; S(m'_{(\lambda)}g_{(\lambda)}^{-\gamma_{2}},\mathbf{g}_{\lambda})}.
\end{multline*}
Again, using the equivalent seminorms from Proposition \ref{chse} and choosing large enough $n_{1}, n_{2}$ we obtain that the above is estimated by
\begin{equation*}
c_{k}\|a\|_{n_{1};S(m,g;\mathfrak{h}^{*})} \|b\|_{n_{2};S(m',g;\mathfrak{h}^{*})},
\end{equation*} 
what ends the proof.
\end{proof}

Let $\Lambda$ denotes a weight $\Lambda(w,\lambda)=|\lambda|$. It is a weight for every weight function on $\mathfrak{h}^{*}$.
\begin{cor} \label{w58}
Suppose that $g$ is a weight function on $\mathfrak{h}^{*}$, and $m$ and $m'$ are $g$-weights.
If $a \in \smgh$ and $b \in \smghp$, then for all $N \in \mathds{N}$
\begin{equation}  \label{w522}
a \# b - \sum_{0 \leq j<N} {1 \over j!}({i \over 2}T_{2}[D_{1,1},D_{1,2}])^{j}(a, b) = R_{N}(a,b),
\end{equation}
where $R_{N}(a,b) \in S(mm'g^{-2N}\Lambda^{N},g;\mathfrak{h}^{*})$. Moreover, for $k, N \in \mathds{N}$ we have
\begin{equation*}
\|R_{N} (a,b)\|_{k;S(mm'g^{-2N}\Lambda^{N},g;\mathfrak{h}^{*})} \leq c_{k,N} \|a \|_{k_{1};S(m,g;\mathfrak{h}^{*})} \|b\|_{k_{2};S(m',g;\mathfrak{h}^{*})}.
\end{equation*}
for some $k_{1}, k_{2} \in \mathds{N}$. If, in addition,
\begin{equation} \label{rats}
\partial^{\alpha}a \in S(m_{1},g;\mathfrak{h}^{*}), \ \ \ \partial^{\beta}b \in S(m_{2},g;\mathfrak{h}^{*}), \ \ \ |\alpha|,|\beta| \geq N,
\end{equation}
then $R_{N}(a,b) \in S(m_{1}m_{2}\Lambda^{N},g;\mathfrak{h}^{*})$ and for $k, N \in \mathds{N}$ there exist $k_{1}, k_{2} \in \mathds{N}$ such that
\begin{equation*}
\|R_{N} (a,b)\|_{k;S(m_{1}m_{2}\Lambda^{N},g;\mathfrak{h}^{*})} \leq c_{k,N} \max_{|\alpha|=|\beta|=N} \|\partial^{\alpha}a \|_{k_{1};S(m_{1},g;\mathfrak{h}^{*})} \|\partial^{\beta}b\|_{k_{2};S(m_{2},g;\mathfrak{h}^{*})}.
\end{equation*}
\end{cor}
\begin{proof}
By Proposition \ref{proz} it is enough to show that the reminder term $R_{N}(a,b)$ given by (\ref{reha})
is in the class $S(mm'g^{-2N}\Lambda^{N},g;\mathfrak{h}^{*})$ with appropriated estimates of the seminorms. We have
\begin{align} \label{prre}
\|R_{N}(a,b)\|_{k;S(mm'g^{-2N}\Lambda^{N},g;\mathfrak{h}^{*})}&
\\ \nonumber
\leq \int_{0}^{1} (1-\theta)^{N-1} {1 \over (N-1)!} \| & ({i \over 2}T_{2}[D_{1,1},D_{1,2}]_{\#_{\theta}})^{N}(a,b) \|_{k;S(mm'g^{-2N}\Lambda^{N},g;\mathfrak{h}^{*})} \ d\theta.
\end{align}
We check that
\begin{align} \nonumber
\|(T_{2}[D_{1,1},D_{1,2}]_{\#_{\theta}})^{N}(a,b) \|_{k;S(mm'g^{-2N}\Lambda^{N},g;\mathfrak{h}^{*})}
&\leq \sum_{|\alpha|=|\beta|=N}|c_{\alpha,\beta,\theta}| \| \partial^{\alpha}a \#_{\theta} \partial^{\beta}b \|_{k;S(mm'g^{-2N},g;\mathfrak{h}^{*})}
\\ \label{kitka}
&\leq c_{N} \max_{|\alpha|=|\beta|=N} \|\partial^{\alpha}a \#_{\theta} \partial^{\beta}b\|_{k;S(mm'g^{-2N},g;\mathfrak{h}^{*})}.
\end{align}
By Leibniz's rule in the case of the group $\mathfrak{h}_{\theta}$ we conclude that Theorem $\ref{compthm}$ is also valid in the case of the groups $\mathfrak{h}_{\theta}$, $\theta \in (0,1)$ with constants $c_{k}$ independent from $\theta$.
Thus, (\ref{kitka}) is estimated by
\begin{equation*}
c_{N,k} \max_{|\alpha|=|\beta|=N} \|\partial^{\alpha}a \|_{k_{1};S(mg^{-N},g;\mathfrak{h}^{*})} \|\partial^{\beta}b \|_{k_{2};S(m'g^{-N},g;\mathfrak{h}^{*})}
\leq c_{N,k,l} \max_{|\alpha|=|\beta|=N} \|a\|_{l;S(m,g;\mathfrak{h}^{*})} \|b\|_{l;S(m',g;\mathfrak{h}^{*})},
\end{equation*}
for large enough $l \in \mathds{N}$. Consequently, (\ref{prre}) is bounded by
\begin{multline*}
\int_{0}^{1} {(1-\theta)^{N-1} \over (N-1)!} c_{N,k,l} \max_{|\alpha|=|\beta|=N} \|a\|_{l;S(m,g;\mathfrak{h}^{*})} \|b\|_{l;S(m',g;\mathfrak{h}^{*})} \ d\theta
\\
\leq c_{N,k,l} \max_{|\alpha|=|\beta|=N} \|a\|_{l;S(m,g;\mathfrak{h}^{*})} \|b\|_{l;S(m',g;\mathfrak{h}^{*})}.
\end{multline*}
In the similar way we get the thesis in the case (\ref{rats}).
\end{proof}

\begin{cor} \label{ogl}
Assume that $g$ is a weight function on $\mathfrak{h}^{*}$. If $A \in \sogd$, then the operator $\Op(A)$ is bounded on $\lth$. Moreover, there exists $k \in \mathds{N}$ such that
\begin{equation*}
\|\Op(A)\|_{\lth \to \lth} \leq C\|A\|_{k;\sogd}. 
\end{equation*}
\end{cor}
\begin{proof}
By Proposition \ref{wpt} we get
\begin{equation*}
\|\Op(A)\|_{\lth \to \lth} = \sup_{\lambda \neq 0} \|\pi^{\lambda}_{A}\|_{\ltrn \to \ltrn}
= \sup_{\lambda \neq 0} \|\Op^{W}(a_{\lambda})\|_{\ltrn \to \ltrn}, 
\end{equation*}
where $a_{\lambda}=\widehat{A}_{(\lambda)}$. Using equivalent seminorms we obtain
\begin{equation*}
\sup_{\lambda \neq 0}\|\Op^{W}(a_{\lambda})\|_{\ltrn \to \ltrn} \leq C\sup_{\lambda \neq 0}\|a_{\lambda}\|_{k;S(1,\mathbf{g}_{\lambda})} \leq C\|A\|_{k;S(1,g;\mathfrak{h})}. 
\end{equation*}
\end{proof}
Let us notice that more sublimated conditions determining $L^{2}$-boundedness of convolution operators in terms of symbols one can find in G{\l}owacki \cite{Glo1}, \cite{Glo5}.


\section{Inverses}
\subsection{Invertibility in the algebra of pseufodifferential operators}
Relation of invertibility of operators in theoretical sense and invertibility as pseudodifferential operators is quite natural.

Let symbols $a, b$ satisfy $a \#_{W} b = b\#_{W}a = 1$, i.e. $b$ is the inverse symbol of $a$. Such symbol will be sometimes denoted by $a^{-1\#_{W}}$. By Leibniz's rule for $a\#_{W}b=1$ we get the following rule of differentiating the inverse
\begin{equation*}
\partial^{\alpha} b = \sum_{p=1}^{|\alpha|} \sum_{\sum_{i=1}^{p} \alpha^{i} = \alpha} c_{\alpha^{1}\ldots\alpha^{p}} b \#_{W} \left(\partial^{\alpha^{1}}a \#_{W} b \right)\#_{W} \ldots \#_{W} \left( \partial^{\alpha^{p}}a \#_{W} b \right).
\end{equation*}

From the Beals characterization one can conclude the following corollary.
\begin{cor} \label{corbealsthm}
The class $S(1,\mathbf{1})$ has inverse-closed property, i.e. if $A=\Op^{W}(a)$, where $a$ belongs to the class $S(1,\mathbf{1})$, is invertible on $L^{2}$, then $A^{-1}=\Op^{W}(b)$ with $b \in S(1,\mathbf{1})$. Moreover,
\begin{equation} \label{csi}
\|b\|_{k;S(1,\mathbf{1})} \leq c_{k,l} \max_{1 \leq p \leq l} 
\|\Op^{W}(b)\|^{p+1}_{L^{2}(\rn) \to L^{2}(\rn)} \|a\|^{p}_{l;S(1,\mathbf{1})}.
\end{equation}
\end{cor}

Now, let us consider diagonal metrics (\ref{diag}).
The classes $S(1,\mathbf{g})$ also have inverse-closed property.
\begin{prop} \label{beod}
Let $a \in S(1,\mathbf{g})$ and let the operator $\Op^{W}(a)$ be invertible on $L^{2}(\mathds{R}^{n})$.
Then $\Op^{W}(a)^{-1}=\Op^{W}(a^{-1 \#_{W}})=\Op^{W}(b)$, where $b \in S(1,\mathbf{g})$. Moreover,
\begin{equation*}
\|a^{-1 \#_{W}}\|_{k;S(1,\mathbf{g})} \leq c_{k,l} \max_{1 \leq p \leq l} 
\|\Op^{W}(b)\|^{p+1}_{L^{2}(\rn) \to L^{2}(\rn)} \|a\|^{p}_{l;S(1,\mathbf{g})}
\end{equation*}
\end{prop}
\begin{proof}
Using that
\begin{equation*}
\|u\|_{k;S(1,\mathbf{g})} = \max_{|\alpha| \leq k} \|\partial^{\alpha}u\|_{0;S(g^{-|\alpha|},\mathbf{1})}, \qquad u \in S(1,\mathbf{g}),
\end{equation*}
we get
\begin{equation} \label{powyz}
\|a^{-1 \#_{W}}\|_{k;S(1,\mathbf{g})} = \max_{|\alpha| \leq k} \|\partial^{\alpha}a^{-1 \#_{W}}\|_{0;S(g^{-|\alpha|},\mathbf{1})}.
\end{equation}
If $a \in S(1,\mathbf{g})$, then $\partial^{\alpha}a \in S(g^{-|\alpha|},\mathbf{1})$, $\alpha \in \mathds{N}^{2n}$. In particular $a \in S(1,\mathbf{1})$. Using invertibility of the operator $\Op^{W}(a)$ we obtain that $a^{-1 \#_{W}} \in S(1,\mathbf{1})$. 
From the rule of differentiating of the inverse
and using the Weyl-H{\"o}rmander symbolic calculus on the Phase Space we get that (\ref{powyz}) is estimated by
\begin{equation*}
\max_{|\alpha| \leq k} \sum_{p;\alpha^{1},\ldots,\alpha^{p}} |c_{\alpha;\alpha^{1}\ldots\alpha^{p}}| \|a^{-1 \#_{W}}\|^{p+1}_{j;S(1,\mathbf{1})}\prod_{1 \leq i \leq p}\|\partial^{\alpha^{i}}a\|_{j(\alpha^{i});S(g^{-|\alpha^{i}|},\mathbf{1})},
\end{equation*}
where $0 \leq p \leq |\alpha|$ and $\sum_{i=1}^{p}|\alpha^{i}|=|\alpha|$.
Using again the estimates
\begin{equation*}
\|\partial^{\alpha}u\|_{k;S(g^{-|\alpha|},\mathbf{1})} \leq \|u\|_{k;S(1,\mathbf{g})}, \qquad u \in S(1,\mathbf{g}),
\end{equation*}
\end{proof}

The operator $A \in S(m,\mathbf{g})$ is defined on every subspace of $\mathcal{S}'(W)$.
For H{\"o}rmander metric $\mathbf{g}$ and $\mathbf{g}$-weight $m$ we define the Sobolev space
\begin{equation*} 
H(m,\mathbf{g})=\{u \in \mathcal{S}'(W): Au \in L^{2}(\mathds{R}^{n}), \ \ A \in \Op^{W} S(m,\mathbf{g})\}.
\end{equation*}
%

Beals \cite{Be3} extended the characterization \ref{bealsthm} for very general class of metrics and weight. It was clarified by Bony \cite{JMB2}, \cite{JMB} and Bony-Chemin \cite{BC} assuming so-called geodesically temperance of metrics. Such condition is satisfied by diagonal metrics.

\begin{thm}
Suppose that $m$, $m'$ are weights for diagonal metric $\mathbf{g}$. If $a \in S(m,\mathbf{g})$ and $\Op^{W}(a)$ is a topological isomorphism $H(m,\mathbf{g})$ onto $H(m'/m,\mathbf{g})$, then the symbol of the inverse operator $\Op^{W}(a)^{-1}$ is in the class $S(m^{-1},\mathbf{g})$.
\end{thm}
From the proof it follows that the seminorms of the inverse symbol depend only on seminorms of the initial symbol, structural constants of metric and norm in $\mathcal{B}(\ltrn)$ of the inverse operator.

Now, we consider the symbols with the parameter.
Let $m$ be a weight for metric $\mathbf{g}$. We say that the family of symbols ${(a_{u})}_{u \in U}$ is bounded in the class $S(m,\mathbf{g})$, if for all $k \in \mathds{N}$, seminorms $\|a_{u}\|_{k;S(m,\mathbf{g})}$ are unifirmly bounded.
\begin{cor}
Let $(a_{u})_{u \in U}$ be a family of symbols bounded in $S(m,\mathbf{g})$. Suppose that the operators $\Op^{W}(a_{u})$ are invertible on $\ltrn$ and moreover, 
the norms of the inverse operators $\Op^{W}(a_{u})^{-1}$ are uniformly bounded with respect to $u \in U$. Then, the family $({a_{u}^{-1\#}})_{u \in U}$ of the inverse symbols is bounded in the class $S(m^{-1},\mathbf{g})$.
\end{cor}

The following lemma (in different context, but the proof is the same) is due to G{\l}owacki \cite{Glo3} (Lemma 4.5). 
\begin{lem} \label{pglem}
Let ${(a_{u})}_{u \in U}$ be a family of symbols from the class $S(m,\mathbf{g})$ depending smoothly on $u \in U$. If the operators with the symbols $a_{u}$ are invertible and the family of the inverse symbols ${a_{u}^{-1\#}}$ is bounded in $S(m^{-1},\mathbf{g})$, then $a_{u}^{-1\#}$ depending smoothlu on the parameter $u \in U$.
\end{lem}

\begin{cor} \label{wn20}
Let ${(a_{u})}_{u \in U}$ be a family of symbols depending smoothly on $u \in U$ and bounded in $S(m,\mathbf{g})$.
Suppose that the operators $\Op^{W}(a_{u})$ are invertible on $\ltrn$ and seminorms of the inverse operators $\Op^{W}(a_{u})^{-1}$ are uniformly bounded with respect to $u \in U$. Then, the family $({a_{u}^{-1\#}})_{u \in U}$ depending smoothly with respect to $u \in U$.
\end{cor}

\subsection{Inverse-closed property on the Heisenberg group}
Before we show the inverse-closed property of the class $\smgh$ we present some useful ideas due to  G{\l}owacki \cite{Glo3}, \cite{Glo6}.

\begin{lem} \label{derin}
Let $A \in \smgd$, $C^{-1} \leq m \leq Cg$, let the operator $\Op(A)$ be invertible on $\lth$ and $\Op(B)=\Op(A)^{-1}$. Then, for all $\alpha \in \mathds{N}^{2n+1}$ we have
\begin{equation*}
T^{\alpha}B
=\sum_{1 \leq p \leq l(\alpha)} \sum_{\sum_{i=1}^{p} l(\alpha^{i})=l(\alpha)} c_{\alpha;\alpha^{1}\ldots\alpha^{p}}
(B*T^{\alpha^{1}}A)*\ldots*(B*T^{\alpha^{p}}A)*B.
\end{equation*}
\end{lem}
\begin{lem} \label{spod}
Let $A \in \smgd$, $C^{-1} \leq m \leq Cg$, let the operator $\Op(A)$ be invertible on $\lth$ and $\Op(B)=\Op(A)^{-1}$.
Then $B$ is an $\sco$-convolver.
\end{lem}

Let $A \in \sogd$ and let the operator $\Op(A)$ be invertible on $\lth$. Let $\Op(B)=\Op(A)^{-1}$ and $a=\widehat{A}$, $b=\widehat{B}$. Then, $a \# b = b \# a =1$ and we will sometimes write $b=a^{-1\#}$ and $B=A^{-1}$.
By Lemma \ref{spod} we get that, $B$ is also $\sco$-convolver and the operator $\pi^{\lambda}_{B}$ is well-defined.

On the other hand, $a_{(\lambda)} \in S(m_{(\lambda)},\mathbf{g}_{\lambda})$ and seminorms are independent from $\lambda$,
where
\begin{equation*}
(\mathbf{g}_{\lambda})_{w}(u) = {|\lambda|\|u\|^{2} \over g(w)^{2}}.
\end{equation*}
The operator $\pi^{\lambda}_{A}=\Op^{W}(a_{(\lambda)})$ is invertible on $\ltrn$. Thus we obtain
\begin{equation*}
\Op^{W}(b_{(\lambda)}) = \pi^{\lambda}_{B}=(\pi^{\lambda}_{A})^{-1}=\Op^{W}(a_{(\lambda)}^{-1\#_{W}}).
\end{equation*}

By Beals-Bony theorem we get that symbols $b_{(\lambda)}$ are in the classes $S(m_{(\lambda)}^{-1},\mathbf{g}_{\lambda})$. To get independence of the symbols $b_{(\lambda)}$ from the parameter $\lambda$ we need uniform invertibility of the operators $\Op^{W}(b_{(\lambda)})$. But it follows from Proposition \ref{wpt}. As a corollary we get
\begin{cor}
If $a \in \smgh$, then $b_{(\lambda)} \in S(m_{(\lambda)}^{-1},\mathbf{g}_{\lambda})$ with seminorms independent from $\lambda$.
\end{cor}

Symbols $b_{(\lambda)}$ are in the classes $S(m_{(\lambda)}^{-1},\mathbf{g}_{\lambda})$, in particular they are smooth (with respect to variable $w \in W$). We want to know if $b_{(\lambda)}$ are also smooth with respect to the parameter $\lambda$. The family $(a_{(\lambda)})_{\lambda \neq 0}$ is smooth with respect to  $\lambda$, as $a_{(\lambda)}(w)=\widehat{A}(\lambda^{1 \over 2}w,\lambda)$, where $A \in \smgd$.

We restrict the parameter $\lambda$ do the set $\Lambda_{k}=({1 \over k},k)$. There exists a positive constant $c_{k}$ such that
\begin{equation*}
c_{k}^{-1} \sup_{\lambda \in \Lambda_{k}}  \mathbf{g}_{\lambda} \leq \mathbf{g}_{1} \leq c_{k} \inf_{\lambda \in \Lambda_{k}} \mathbf{g}_{\lambda},
\end{equation*}
and moreover 
\begin{equation*}
c_{k}^{-1} \sup_{\lambda \in \Lambda_{k}}  m_{(\lambda)} \leq m_{(1)} \leq c_{k} \inf_{\lambda \in \Lambda_{k}} m_{(\lambda)}.
\end{equation*}
In particular for all $\lambda \in \Lambda_{k}$ the family $(a_{(\lambda)})_{\lambda \in \Lambda_{k}}$ is bounded in $S(m_{(1)},\mathbf{g}_{1})$. It follows from Corollary \ref{wn20} that the family of the inverse symbols $(b_{(\lambda)})_{\lambda \in \Lambda_{k}}$ is bounded in $S(m_{(1)}^{-1},\mathbf{g}_{1})$ and depends smoothly on the parameter $\lambda$.
The number $k$ can be chosen arbitrarily and thus the family $(b_{(\lambda)})_{\lambda \neq 0}$ smoothly depends on the parameter $\lambda$.

In the othr hand, $\pi^{\lambda}_{B}=\Op^{W}(\widehat{B}_{(\lambda)})$ and thus,
\begin{equation*}
\widehat{B}(w,\lambda) = b_{(\lambda)}(\lambda^{-{1 \over 2}}w), \ \ \ \lambda \neq 0.
\end{equation*}
Then we get
\begin{cor}
Suppose that $A \in \smgd$, the operator $\Op(A)$ is invertible on $\lth$ and $\Op(B)=\Op(A)^{-1}$. Then $\widehat{B}$ is smooth.
\end{cor}
Actually, we show the smoothness of $\widehat{B}$ outside the set $\lambda \neq 0$. We can assume that $\widehat{B}$ is smooth everywhere, what follows from Sobolev lemma.
\begin{lem}
Suppose that the (distributional) derivatives of $b$ of any order agrees with functions in $L^{2}_{\loc}(\mathds{R}^{2n+1})$. Then, there exists a smooth function $b'$ such that $b(x)=b'(x)$ almost everywhere. 
\end{lem}


Now, we show that the classes $\sogd$ have inverse-closed property. It will be the first step in showing Beals-Bony-type inverse theorem for wider class of weights.
The main tool is the characterization of the class $\sogh$ by using the seminorms of the class of symbols on the Phase Space indexed by the parameter $\lambda$ as well as the inverse-closed property of pseudodifferential operators.


If $B$ is an $\sco$-convolver, then also $T^{\alpha}B$ and $\alpha \in \mathds{N}^{2n+1}$. Thus, the operators
\begin{equation*}
\pi^{\lambda}_{T^{\alpha}B}=\Op^{W}((\partial^{\alpha}b)_{(\lambda)}), \qquad \alpha \in \mathds{N}^{2n+1}.
\end{equation*}
are well-defined.

\begin{thm}
If $A \in \sogd$ and if the operator $\Op(A)$ is invertible on $L^{2}(\mathfrak{h})$, then $B$ is in the class $\sogd$, where $\Op(B)=\Op(A)^{-1}$. Moreover for every $k$ there exists $l$ such that
\begin{equation} \label{eses}
\|a^{-1\#}\|_{k;\sogh} \leq c_{k,l} \max_{1 \leq p \leq l} 
\|\Op(B)\|^{p+1}_{\lth \to \lth} \|a\|^{p}_{l;\sogh}.
\end{equation}
\end{thm}
\begin{proof}
Let $a=\widehat{A}$ and $b=\widehat{B}$. Using equivalent seminorms we get
\begin{equation*}
\|b\|_{k;\sogh} = \max_{k_{1}+k_{2} \leq k} \sup_{\lambda \neq 0} \|(\partial^{k_{2}}_{2}b)_{(\lambda)}\|_{k_{1};S(g_{(\lambda)}^{-k_{2}},\mathbf{g}_{\lambda})}.
\end{equation*}
If $a \in \sogh$, then $(\partial^{\alpha^{i}}a)_{(\lambda)} \in S(g_{(\lambda)}^{-|\alpha^{i}|},\mathbf{g}_{\lambda})$. In particular $a_{(\lambda)} \in S(1,\mathbf{g}_{\lambda})$ and
$b_{(\lambda)} \in S(1,\mathbf{g}_{\lambda})$. By the symbolic calculus and the formula from Corollary \ref{derin} we estimate the above by
\begin{multline*}
\max_{k_{1}+k_{2} \leq k} \sup_{\lambda \neq 0} \sum_{\substack{\{s,\alpha^{1},\ldots,\alpha^{s}: s \leq l(\alpha); \\ \sum_{i=1}^{s}l(\alpha^{i})=2k_{2}\}}} |c_{\alpha;\alpha^{1}\ldots\alpha^{p}}| \|b_{(\lambda)}\|^{s+1}_{j;S(1,\mathbf{g}_{\lambda})} \prod_{1 \leq i \leq s}\|(\partial^{\alpha^{i}}a)_{(\lambda)}\|_{j_{i}(k_{1});S(g_{(\lambda)}^{-|\alpha^{i}|},\mathbf{g}_{\lambda})}
\\
\leq c_{k} \max_{1 \leq s \leq m} \sup_{\lambda \neq 0} \|b_{(\lambda)}\|^{s+1}_{m;S(1,\mathbf{g}_{\lambda})}, \|a_{(\lambda)}\|^{s}_{m;S(1,\mathbf{g}_{\lambda})}.
\end{multline*}
for some $m \in \mathds{N}$. By Proposition \ref{beod} we get that the above is bounded by
\begin{equation} \label{bogf}
c_{k} \max_{1 \leq r, s \leq m} \sup_{\lambda \neq 0} \|\Op^{W}(a_{(\lambda)})^{-1}\|^{(s+1)(r+1)}_{L^{2}(\rn) \to L^{2}(\rn)} \|a_{(\lambda)}\|^{(s+1)r+s}_{m;S(1,\mathbf{g}_{\lambda})},
\end{equation}
maybe at the cost of the number $m$. It follows from Proposition \ref{wpt} that 
\begin{equation*}
\sup_{\lambda \neq 0} \|\Op^{W}(a_{(\lambda)})^{-1}\|_{L^{2}(\rn) \to L^{2}(\rn)}
\leq \|\Op(B)\|_{\lth \to \lth}
\end{equation*}
and thus, again using the equivalent seminorms, (\ref{bogf}) is estimated for some $l \in \mathds{N}$ by
\begin{equation*}
c_{k,l} \max_{1 \leq p \leq l} \|\Op(B)\|^{p+1}_{\lth \to \lth} \|a\|^{p}_{l;\sogh}
\end{equation*}
for some $p \in \mathds{N}$.
\end{proof}

Strictly speaking, estimates of the derivatives of the inverse symbol $b$
\begin{equation*}
|\partial^{\alpha}_{w} \partial^{\beta}_{\lambda} b(w,\lambda)| \leq c_{\alpha,\beta} g(w,\lambda)^{-|\alpha|-\beta}, \qquad \alpha \in \mathds{N}^{2n}, \beta \in \mathds{N}
\end{equation*}
we obtain initially outside the set $\lambda = 0$. The derivatives of $b$ of arbitrary order are locally in  $\lth$. Similarly we have for $\partial_{1}^{\alpha}\partial_{2}^{\beta}b$, where $\alpha \in \mathds{N}^{2n}, \beta \in \mathds{N}$. By Sobolev lemma we can assume that estimates of the symbol $b$ are satisfied in fact for all $(w,\lambda) \in \mathfrak{h}^{*}$ with the same constants.

The same remark remains in force in the case of arbitrary weight $m$.


\subsection{Beals-type theorem}

In the following subsection we partially use the ideas of Beals \cite{Be1}.

If $m$ is a weight for a weight function $g$ on $\mathfrak{h^{*}}$, then we define
\begin{equation*}
\hmgh = \{v \in \scop: \Op(A)v \in \lth, A \in \smgd\}.
\end{equation*}
The topology on $\hmgh$ is given by the family of seminorms $\mathcal{N}_{A}(v)=\|Av\|_{\lth}$, $A \in \smgd$. From the definition we have $\sco \subset \hmgh \subset \scop$.

Let $m \geq 1$.  
The convolution operator with a distribution $A \in \smgd$ which is invertible on $\lth$ and satisfies $\Op(B)=\Op(A)^{-1}$, where $B \in \smigd$ is called a \textit{maximal operator}.
\begin{prop}
If $\Op(A)$ is a maximal operator and $A \in \smgd$, then for every $T \in \smgd$ we have
\begin{equation*}
\|\Op(T)f\|_{\lth} \leq C \|\Op(A)f\|_{\lth}, \ \ \ \ \ f \in \sco.
\end{equation*}
In particular $\hmgh$ is the domain of the operator $\Op(A)$.
\end{prop}
\begin{proof}
If $\Op(B)\Op(A)=\Op(A)\Op(B)=I$, then for $f \in \sco$ we have
\begin{equation*}
\|\Op(T)f\|_{\lth} = \|\Op(T)\Op(B)\Op(A)f\|_{\lth} \leq \|\Op(T)\Op(B)\|_{\lth \to \lth}\|\Op(A)f\|_{\lth}.
\end{equation*}
On the other hand, $T \in \smgd$, $B \in \smigd$, so $T*B \in \sogd$ and by Corollary \ref{ogl} we have for some $k$ the following estimation
\begin{equation*}
 \|\Op(T)f\|_{\lth} \leq c_{k}\|T*B\|_{k; \sogd} \|\Op(A)f\|_{\lth}.
\end{equation*}
\end{proof}


We say that a weight $m$ satisfies elliptic condition, if
$C^{-1} \leq m \leq Cg$ and 
there exists a symbol $u \in \smgh$ such that $|u| \geq C^{-1}m$. Without loss of generality we assume that $u$ is positive and let $U=u^{\vee}$.

Let us notice that the function
\begin{equation*}
l(w,\lambda) = 1+\log(1+\|w\|^{2}+|\lambda|^{2}),
\end{equation*}
satisfies elliptic condition. It is a weight for the weight function $\rho$, where
$\rho(w,\lambda) = (1+\|w\|^{2}+|\lambda|^{2})^{{1 \over 2}}$.

\begin{lem} \label{l418}
If $m$ is a weight for $g$ and it satisfying ellipticity condition, then there exist families $u_{N} \in \smgh$, $v_{N} \in \smigh$ such that
\begin{equation*}
\lim_{N \to \infty} u_{N} \# v_{N} \to 1
\end{equation*}
in $S(1,g;\mathfrak{h}^{*})$.
\end{lem}
\begin{proof}
Let $u \in \smgh$ be a symbol comparable with the weight $m$, and let $\psi$ be a bump function.
We define
\begin{equation*}
u_{N}(\xi)=N\psi(N^{-1}u(\xi))+u(\xi)(1-\psi(N^{-1}u(\xi))), \qquad \xi \in \mathds{R}^{2n+1}.
\end{equation*}
Then $u_{N} \in \smgh$ (the seminorms can depend on $N$). Moreover, for $|\alpha| \geq 1$ the family $(N\partial^{\alpha}u_{N})_{N \in \mathds{N}}$ is bounded in $S(mg,g;\mathfrak{h}^{*})$ uniformly with respect to $N$.

Let $v_{N}=u_{N}^{-1}$ (the inverse in sense of ordinary multiplication). By Fa\'a di Bruno formula and $u_{N} \in \smgh$ we get that $v_{N} \in \smigh$. We also have the following estimations
\begin{equation} \label{fdbi}
\|v_{N}\|_{k;\smigh} \leq c_{N,k}\|u_{N}\|_{k;\smgh}. 
\end{equation}
Moreover, for $|\alpha| \geq 1$ the family $(N\partial^{\alpha}v_{N})_{N \in \mathds{N}}$ is bounded in  $S(m^{-1}g,g;\mathfrak{h}^{*})$ uniformly with respect to $N$. By the asymptotic form (\ref{w522}) of the length $1$ for $u_{N} \# v_{N}$ we get that
\begin{equation} \label{rasl}
u_{N} \# v_{N} - 1 =N^{-2}r(u_{N},v_{N}),
\end{equation}
where $r(u_{N},v_{N}) \in \sogh$ uniformly with respect to $N$, what ends the proof. In the similar way one can show that
\begin{equation*}
\lim_{N \to \infty} v_{N} \# u_{N} \to 1
\end{equation*}
in $S(1,g;\mathfrak{h}^{*})$.
\end{proof}

\begin{prop} \label{pr73}
If $m$ is a weight for $g$ satisfying ellipticity condition, then there is a maximal operator $\Op(S)$, such that $S \in \smgd$. Moreover, for every $k \in \mathds{N}$ there is $l \in \mathds{N}$ such that
\begin{equation*}
\|\sigma^{-1\#}\|_{k;\smigh} \leq c_{k,l} \max_{1 \leq p \leq l} \|\sigma\|^{2p+1}_{l;\smgh},
\end{equation*}
where $\sigma=\widehat{S}$.
\end{prop}
\begin{proof}
Let $u_{N} \in \smgh$, $v_{N} \in \smigh$ be the symbols from the previous lemma and let $U_{N}=u_{N}^{\vee}$, $V_{N}=v_{N}^{\vee}$. For large enough $N_{0}$ the operators
\begin{equation*}
\Op(U_{N_{0}})\Op(V_{N_{0}}), \ \ \ \ \ \Op(V_{N_{0}})\Op(U_{N_{0}})
\end{equation*}
are invertible. Let $S=U_{N_{0}}$. Then $S \in \smgd$ and $\Op(S)$ is a surjection with the continuous right inverse
\begin{equation*}
\Op(V_{N_{0}})(\Op(S)\Op(V_{N_{0}}))^{-1}
\end{equation*}
and it is injective with the continuous left inverse
\begin{equation*}
(\Op(V_{N_{0}})\Op(S))^{-1}\Op(V_{N_{0}}).
\end{equation*}
By the symbolic calculus it follows that $S^{-1} \in \smigd$. We also have
\begin{equation*}
\|\sigma^{-1\#}\|_{k;\smigh} = \|v_{N_{0}} \# (\sigma \# v_{N_{0}})^{-1\#}\|_{k;\smigh}.
\end{equation*}
By the symbolic calculus, Proposition \ref{beod} and $(\ref{fdbi})$ there exists $l \in \mathds{N}$ such that
\begin{equation*}
\|\sigma^{-1\#}\|_{k;\smigh} \leq  c_{k,l} \max_{1 \leq p \leq l} \|\sigma\|^{2p+1} \|\Op(U_{N_{0}})\Op(V_{N_{0}})\|^{p+1}_{\lth \to \lth},
\end{equation*}
what, at the cost of the constant $c_{k,l}$, gives the thesis.
\end{proof}


\begin{thm} \label{tebi}
Let $m$ be a weight for a weight function $g$ on $\mathfrak{h}^{*}$ and let $m$ satisfies ellipticity condition. If $A \in \smgh$ and $\Op(A)$ is invertible on $L^{2}(\mathfrak{h})$, $\Op(B)=\Op(A)^{-1}$, then $B$ is in the class $\smigh$.
\end{thm}
\begin{proof}
Let $\sigma$ be the symbol of the maximal operator $S$. By the symbolic calculus
\begin{equation*}
\|a^{-1\#}\|_{k;S(m^{-1},g;\mathfrak{h}^{*})} = \|a^{-1\#} \# \sigma \# \sigma^{-1\#}\|_{k;S(m^{-1},g;\mathfrak{h}^{*})}
\leq c_{k} \|a^{-1\#} \# \sigma \|_{k_{1};S(1,g;\mathfrak{h}^{*})}
\|\sigma^{-1\#}\|_{k_{2};S(m^{-1},g;\mathfrak{h}^{*})},
\end{equation*}
for some $k_{1}$, $k_{2}$.
\end{proof}


\subsection{Resolvent}
We investigate the resolvents of the pseudodifferential operators. Some results of the previous subsection are little modified, because we work with the parameter-dependent symbols.

\begin{lem}
Let $z \in \mathds{C}$ and $m$ be a weight for a weight function $g$ on $\mathfrak{h}^{*}$. Then the functions
\begin{equation*}
m_{|z|} = |z|+m
\end{equation*}
are weights for the weight function $g$. Moreover the structural constants are independent from $z$.
\end{lem}

Let $a \in \smgh$. Then $a_{z}:=z+a$ is in the class $\smzgh$ with the seminorms independent from $z$. Moreover, for all $|\alpha| \geq 1$ we have $|\partial^{\alpha}a_{z}|=|\partial^{\alpha}a|$ and thus, $\partial^{\alpha}a_{z} \in S(mg^{-|\alpha|},g;\mathfrak{h}^{*})$ uniformly with respect to $z$.

Notice that if $m$ satisfy elliptic condition then $m_{|z|}$ also.
However, the constant $C$ in the condition $m_{|z|} \leq Cg$ can be depend on $z$ and we need to be careful. 

We restrict to the complex numbers with the sector
\begin{equation*}
\Sigma_{\psi_{0}} = \{z \in \mathds{C}: |\arg(z)| < \psi_{0} \}.
\end{equation*}
Notice that then $\Rez(z) \geq \cos(\psi_{0})|z|$ and thus for all $a \geq 0$ we have
\begin{equation*}
|z|+a \leq \left({2 \over 1+\cos^{2}(\psi_{0})}\right)^{{1 \over 2}} |z+a|.
\end{equation*} 

\begin{lem}
If $m$ is an elliptic weight for $g$, then there exist families $u_{N,z} \in \smzgh$, $v_{N,z} \in \smzigh$ uniformly with respect to $z \in \Sigma_{\psi_{0}}$ such that
\begin{equation*}
\lim_{N \to \infty} u_{N,z} \# v_{N,z} \to 1
\end{equation*}
in $S(1,g;\mathfrak{h}^{*})$ uniformly.
\end{lem}
\begin{proof}
We define $u_{N,z}=z+u_{N}$, where $u_{N}$ is as in Lemma \ref{l418}.
Then $u_{N,z} \in \smzgh$ and the seminorms can depend on $N$, but are independent from $z$. Moreover, for $|\alpha| \geq 1$ we have
$N\partial^{\alpha}u_{N,z} \in S(m,g;\mathfrak{h}^{*})$ uniformly with respect to $N$ and $z$.

The further part of the proof is the same as in Lemma \ref{l418}.

\end{proof}

\begin{prop} \label{p77}
If $m$ is elliptic weight for $g$, then there exists a maximal operator $\Op(S(z))$, where $S(z) \in \smzgd$ with the seminorms independent from $z$. Moreover, the inverse operator satisfies
\begin{equation*}
\|\sigma(z)^{-1\#}\|_{k;\smzigh} \leq c_{k,l} \max_{1 \leq p \leq l} \|\sigma(z)\|^{2p+1}_{l;\smzgh},
\end{equation*}
where $\sigma(z)=\widehat{S}(z)$.
\end{prop}
\begin{proof}
Let $u_{N,z} \in \smzgh$, $v_{N,z} \in \smzigh$ be the symbols from the previous lemma and let $U_{N,z}=u_{N,z}^{\vee}$, $V_{N,z}=v_{N,z}^{\vee}$. Then, for large enough $N_{0}$, the operators
\begin{equation*}
\Op(U_{N_{0},z})\Op(V_{N_{0},z}), \ \ \ \ \ \Op(V_{N_{0},z})\Op(U_{N_{0},z})
\end{equation*}
are invertible. The thing is the number $N_{0}$ can be chosen independently form $z$, because $r(u_{N,z},v_{N,z}) = r(u_{N},v_{N})$. The further reasoning in identical as in Proposition \ref{pr73}.
\end{proof}


\begin{cor}
Suppose that $m$ is an elliptic weight for a weight $g$ and $\Op(A) \in \smgd$ is invertible. Also let $z$ be a number from the resolvent set of the operator $\Op(A)$ and simultaneously $z \in \Sigma_{\psi}$, $\psi > {\pi \over 2}$. Suppose that $\|R_{z}\|_{\lth \to \lth} \leq M|z|^{-1}$ for some $M>0$. Then $\Op(A_{z})=zI-\Op(A)$ is a maximal operator and $\Op(A_{z})\smzgd$. In particular, $R_{z} \in \smzigd$ with the seminorms independent from $z$.
\end{cor}
\begin{proof}
By Theorem \ref{tebi} we get that $\Op(A)$ is a maximal operator and $A \in \smgd$.
Let $S(z)$ be a maximal operator, $S(z) \in \smzgd$, from Proposition\ref{p77}, and let $\sigma_{z}$ be its symbol.
In particular, we obtain $\sigma_{z}\in S(m_{|z|},g;\mathfrak{h}^{*})$, and $\sigma_{z}^{-1\#} \in S(m_{|z|}^{-1},g;\mathfrak{h}^{*})$. The seminorms are independent from $z$.
By the symbolic calculus we get
\begin{equation*}
\|r_{z}^{-1\#}\|_{k;S(m_{|z|}^{-1},g;\mathfrak{h}^{*})} = \|\sigma_{z}^{-1\#} \# \sigma_{z} \# r^{-1\#}\|_{k;S(m_{|z|}^{-1},g;\mathfrak{h}^{*})}
\leq c_{k} \|\sigma_{z} \# r_{z}^{-1\#}\|_{k_{1};S(1,g;\mathfrak{h}^{*})}
\|\sigma_{z}^{-1\#}\|_{k_{2};S(m_{|z|}^{-1},g;\mathfrak{h}^{*})},
\end{equation*}
for soem $k_{1}$, $k_{2}$.
Using by the estimates of the seminorms of the inverse as in (\ref{eses}), we estimate $(a_{z} \# \sigma_{z}^{-1\#})^{-1\#}$ by
\begin{equation*}
c_{k,j} \max_{1 \leq s \leq j} \|a_{z} \# \sigma_{z}^{-1\#}\|^{s}_{j;S(1,g;\mathfrak{h}^{*})} \|(zI-\Op(S))R_{z}\|_{\lth \to \lth}.
\end{equation*}
We check that
\begin{align*}
\|(zI-\Op(S))R_{z}\|_{\lth \to \lth}
&\leq |z|\|R_{z}\|_{\lth \to \lth} + \|-\Op(S)R_{z}\|_{\lth \to \lth}
\\
&\leq M+\|-\Op(A)(zI+\Op(A))^{-1}\|_{\lth \to \lth} \leq 2M+1.
\end{align*}
The above estimates give the thesis.
\end{proof}


If $A$ is a symmetric distribution on $\mathfrak{h}$, then it is a symmetric distribution on each $\mathfrak{h}_{\theta}$, $\theta \in [0,1]$. Consequently, the operators $\Op_{\theta}(A)f:=A*_{\theta}f$ are selfadjoint on  $L^{2}(\mathds{R}^{2n+1})$. 
Let $B_{z}^{\theta}$ be the resolvents for $A$ with respect to the convolution $*_{\theta}$. In particular we have
\begin{equation*} 
B_{z}^{0}*_{0}(z\delta-P)=(z\delta-P)*_{0}B_{z}^{0}=\delta, \ \ \ B_{z}^{1}*(z\delta-P)=(z\delta-P)*B_{z}^{1}=\delta.
\end{equation*}
From the proof of the previous corollary we get that $B_{z}^{\theta}$, $\theta \in [0,1]$ are in the classes $S((|z|+m)^{-1},\rho;\mathfrak{h}_{\theta})$ uniformly with respect to $\theta$ (in particular, they are $\mathcal{S}$-convolvers on the groups $\mathfrak{h}_{\theta}$).
We denote $b_{z,\theta}=\widehat{B_{z,\theta}}$. In particular, $b_{z}^{0}= (z-a)^{-1}$ (the inverse in the sense of common multiplication). By Beals Theorem $b_{z,\theta}$ is in the class $\smzigd$ and the seminorms are independent from $z$ and $\theta$. We obtain decomposition of $b_{z}^{1}$ in terms of (simplest) symbols $b_{z}^{0}$. 
\begin{prop} \label{rzlr}
We have
\begin{equation*}
b_{z}^{1}=b_{z}^{0}+H_{z},
\end{equation*}
where $H_{z} \in S({m \over (m+|z|)^{2}}g^{-2},g;\mathfrak{h}^{*})$.
\end{prop}
\begin{proof}
Let $a_{z}=z+a$. Then $a_{z} \in S(m_{|z|},g;\mathfrak{h}^{*})$, but $\partial^{\alpha}a_{z} \in S(mg^{-|\alpha|},g;\mathfrak{h}^{*})$. Moreover, $b_{z}^{0} \in S(m_{|z|}^{-1},g;\mathfrak{h}^{*})$ and by Fa\`a di Bruno formula we get $\partial^{\alpha}b_{z}^{0} \in S({m \over m_{|z|}^{2}}g^{-|\alpha|},g;\mathfrak{h}^{*})$. We have the following decompose of $a_{z} \# b_{z}^{0}$
\begin{equation*}
a_{z} \# b_{z}^{0} = a_{z}b_{z}^{0}
+ S_{1}(a_{z},b_{z}^{0}) 
+  R_{2}(a_{z},b_{z}^{0}),
\end{equation*}
where
\begin{equation*}
S_{1}(a_{z},b_{z}^{0}) = {i\lambda \over 2} \sum_{j=1}^{n} (\partial_{1,1,j}a_{z}\partial_{1,2,j}b_{z}^{0}-\partial_{1,2,j}a_{z}\partial_{1,1,j}b_{z}^{0}).
\end{equation*}
Moreover, $R_{2}(a_{z},b_{z}^{0}) \in S({m^{2} \over m_{|z|}^{2}}\Lambda^{2}g^{-4},g;\mathfrak{h}^{*})$.
and we have $a_{z}b_{z}^{0}=1$. Moreover, for $|\alpha|=1$ we have $\partial^{\alpha}a_{z}=\partial^{\alpha}a$ and $\partial^{\alpha}b_{z}^{0}=(b_{z}^{0})^{2}\partial^{\alpha}a$ and thus $S_{1}(a_{z},b_{z}^{0})=0$. Then,
\begin{equation*}
a_{z} \# b_{z}^{0} = 1 +  R_{2}(a_{z},b_{z}^{0}).
\end{equation*}
Acting by $\#$ with $b_{z}^{1}$ we obtain
\begin{equation*}
b_{z}^{1} = b_{z}^{0} - b_{z}^{1} \# R_{2}(a_{z},b_{z}^{0}).
\end{equation*}
Iterating the above expression we get
\begin{equation*}
b_{z}^{1} = b_{z}^{0} - (b_{z}^{0} - b_{z}^{1} \# R_{2}(a_{z},b_{z}^{0})) \# R_{2}(a_{z},b_{z}^{0}).
\end{equation*}
By the symbolic calculus we get that $b_{z}^{0} \# R_{2}(a_{z},b_{z}^{0})$ is in the class $S({m^{2} \over m_{|z|}^{3}}\Lambda^{2}g^{-4},g;\mathfrak{h}^{*})$. On the other hand, $b_{z}^{1} \# R_{2}(a_{z},b_{z}^{0}) \# R_{2}(a_{z},b_{z}^{0})$ is in the class $S({m^{4} \over m_{|z|}^{5}}\Lambda^{4}g^{-8},g;\mathfrak{h}^{*})$ and finally,
\begin{equation*}
b_{z}^{1} - b_{z}^{0} \in S({m^{2} \over m_{|z|}^{3}}\Lambda^{2}g^{-4},g;\mathfrak{h}^{*}).
\end{equation*}
\end{proof}


\section{Semigroups of measures} \label{s4}
\subsection{Generalised laplacians and Hunt theory}

Let $G$ be a (connect, simply connect, nilpotent) Lie group $G=(\mathds{R}^{d},\circ)$ with the neutral element $e$ and let  $\mathcal{M}(G)$ denotes the set of probabilistic Borel measures on $G$. The convolution of measures $\mu_{1}, \mu_{2} \in \mathcal{M}(G)$ is defined by
\begin{equation*}
\int_{G} f(x) (\mu_{1}*\mu_{2})(dx) = \int_{G}\int_{G} f(x \circ y) \mu_{1}(dx) \mu_{2}(dy), 
\end{equation*}
for real bounded Borel functions on $G$. \textit{Reversed measure} $\widetilde{\mu}$ associated to the measure $\mu \in \mathcal{M}(G)$ is defined by $\widetilde{\mu}(A)=\mu(A^{-1})$. A measure $\mu \in \mathcal{M}(g)$ is called symmetric if $\mu=\widetilde{\mu}$.
We say that a semigroup is symmetric if each $\mu_{t}$ is. Notice that $\mu_{t}$ is a convolution semigroup of measures iff $\widetilde{\mu_{t}}$ is.

If $P$ is a generalized laplacian, then for every bump function $\varphi \in C_{c}^{\infty}(G)$, the distribution $(1-\varphi)P$ is a positive bounded measure, so $P$ can be decompose as a sum $P=P_{0}+\mu$, where $P_{0}$ is a distribution with compact support and $\mu$ a bounded positive measure.
In particular $P$ is tempered distribution. The formula
\begin{equation*} 
\langle \nu, f \rangle = \langle P, f \rangle, \ f \in C_{c}^{\infty}(G \backslash \{e\}),
\end{equation*}
defined the L{\'e}vy measure of the functional $P$. 
Moreover, every generalized laplacian extends to continuous linear functional on $C^{2}(G)$ and such extension has a property (\ref{max}) (see e.g. Faraut \cite{Far}, Proposition IV.1).

A semigroup of measures $(\mu_{t})_{t>0}$ determines the strongly continuous semigroup of linear operators of contraction $C_{0}(G)$ by
\begin{equation}
T_{t}f(x)=\mu_{t}*f(x) = \int_{G}f(y^{-1} \circ x)\mu_{t}(dy).
\end{equation}
The infinitesimal generator of semigroup $T_{t}$ is denoted by $\mathcal{L}$ and called the \textit{Hunt generator}. The semigroup $T_{t}$ is called the \textit{Hunt semigroup}.


Let $(X_{j})_{1 \leq j \leq d}$ be a base of left-invariant vector fields. We define
\begin{equation*}
C^{2}(G) = \{f \in C_{0}(G): X_{i}f \in C_{0}(G), X_{i}X_{j}f \in C_{0}(G), 1 \leq i, j \leq d\}.
\end{equation*}

Let us introduce a notion of coordinate functions. Let $\Phi_{j} \in C_{c}^{\infty}(G)$ be real functions such that
\begin{equation}
X_{i}\Phi_{j}(e)=\delta_{i,j}, \ \ \ 1 \leq  i, j \leq d.
\end{equation}
There exist a smooth function $\Phi^{2}$ which is $[0,1]$-valued such that
\begin{enumerate}[\upshape(i)] 
\item $\Phi^{2} = \sum_{i=1}^{d} \Phi_{i}^{2}$ in neighborhood of $e$,
\item $\Phi^{2} = 1$ outside a compact neighborhood of $e$,
\item $\Phi^{2} > 0$, \qquad $x \neq e$.
\end{enumerate}
The function $\Phi^{2}$ is called \textit{Hunt function}.

A measure $\nu$ on Borel subsets of $G \backslash \{e\}$ is called a \textit{L{\'e}vy measure}, if
\begin{equation*}
\int_{G \backslash \{e\}} \{1 \wedge \Phi^{2}(y) \} \nu(dy) < \infty.
\end{equation*}

\begin{thm}
Let $(\mu_{t})_{t>0}$ be a convolution semigroup of measures on $G$ with the Hunt generator $\mathcal{L}$. Then the domain
of $\mathcal{L}$ contains $C^{2}(G)$ and for all $\sigma \in G$, $f \in C^{2}(G)$
\begin{equation*}
\mathcal{L}f(x) = b^{i}X_{i}f(x)+a^{ij}X_{i}X_{j}f(x)+\int_{G \backslash \{0\}} (f(y^{-1} \circ x)-f(x)-\Phi_{i}(y)X_{i}f(x)) \nu(dy),
\end{equation*}
where $b=(b^{1},\ldots,b^{d}) \in G$, $a=(a^{ij})$ is a negative definite symmetric real matrix $d \times d$, and $\nu$ is a L{\'e}vy measure on $G - \{0\}$.
\end{thm}

Let us notice that the Hunt generator $\mathcal{L}$ is related to the generating functional $P$ of semigroup $\mu_{t}$ by $\langle P, f \rangle = \mathcal{L}f(0)$ and
\begin{equation*}
\mathcal{L}f(x)=P*f(x), \qquad f \in C_{c}^{\infty}(G).
\end{equation*}

The semigroup $T_{t}$ restricts to $C_{c}^{\infty}(G)$ extends to the bounded operator on $L^{2}(G)$ which will be denote in the same way. The semigroup $(T_{t})_{t>0}$ is then a strongly continuous semigroup of contractions on $L^{2}(G)$. From now on we will consider the Hunt semigroup on $L^{2}(G)$. It is not hard to see that $\widetilde{T_{t}}=T_{t}^{*}$ and $\widetilde{\mathcal{L}}=\mathcal{L}^{*}$.
Moreover, selfadjointess of the infinitesimal generator $\mathcal{L}$ is equivalent to symmetricity of $P$.
In that case, the semigroup $T_{t}$ can be extended to an analytic semigroup.

We recall some aspects of theory of analytic semigroups from Pazy \cite{Paz}.
Let $\Sigma = \{z \in \mathds{C}: \varphi_{1} < \arg z < \varphi_{2}, \varphi_{1} < 0 < \varphi_{2}\}$, i.e. $\Sigma$ is a \textit{sector}. We say that a semigroup $T_{z}$ is analytic in the sector $\Sigma$, if $z \mapsto T_{z}$ is analytic in $\Sigma$ and $T_{z_{1}}T_{z_{2}}=T_{z_{1}+z_{2}}$ for $z_{1}, z_{2} \in \Sigma$, $T_{0}=I$ and moreover,
\begin{equation*}
\lim_{\substack{z \to 0 \\ z \in \Sigma}} T_{z}x \to x, \qquad x \in X.
\end{equation*}
We say that a semigroup of operators $T_{t}$ is analytic, if it is analytic in some sector containing the nonnegative ray. Analyticity of the semigroup in the sector $\Sigma_{\varphi} = \{z \in \mathds{C}: -\varphi < \arg z < \varphi\}$, $\varphi \leq {\pi \over 2}$ is equivalent to the fact that $\Sigma_{{\pi \over 2}+\varphi} \subset \rho(A)$ and
\begin{equation*}
\|R_{z}\| \leq M|z|^{-1}, \ \ \ z \in \Sigma_{{\pi \over 2}+\varphi}.
\end{equation*}
In that case we can inverse the Laplace transform.
\begin{prop} \label{classy}
Let $R_{z}=R_{z}(A)$ be the resolvent of $A$ and let
\begin{equation*}
\|R_{z}\| \leq M|z|^{-1}, \ \ \ z \in \Sigma_{{\pi \over 2}+\varphi}.
\end{equation*}
Then,
\begin{equation*} 
T_{t} = \int_{\Gamma} e^{zt} R_{z} \ dz,
\end{equation*}
where $\Gamma$ is a smooth curve given by
$$\Gamma= \begin{cases}
re^{i\varphi_{0}} & \mbox{dla } r \geq t^{-1},\cr
t^{-1}e^{i\varphi} & \mbox{dla } -\varphi_{0} \leq \varphi \leq \varphi_{0},\cr
re^{-i\varphi_{0}} & \mbox{dla } r \geq t^{-1}, \end{cases}$$
for every ${\pi \over 2} < \varphi_{0} < \pi$.
\end{prop}

We prove the following estimations.
\begin{lem} \label{lga}
Let $a >0$. Then
\begin{equation*} 
|\int_{\Gamma} e^{zt} (a+|z|)^{-3} \ dz| \leq Ct^{2}(1+ta)^{-3}
\end{equation*}
\end{lem}
\begin{proof}
We analyze the integral on each part of the curve $\Gamma$. Let $\Gamma_{1}(r)=re^{i\varphi_{0}}$ for $r \geq t^{-1}$. Then,
\begin{align*} 
|\int_{\Gamma_{1}} e^{zt} (a+|z|)^{-3} \ dz| 
&= |\int_{t^{-1}}^{\infty} e^{tre^{i\varphi_{0}}}e^{i\varphi_{0}}(a+r)^{-3} dr|
\leq (a+t^{-1})^{-3} \int_{t^{-1}}^{\infty} e^{tr\cos(\varphi_{0})} dr
\\
&\leq {t^{2} \over (1+at)^{3}} \int_{1}^{\infty} e^{s\cos(\varphi_{0})} ds
= c_{\varphi_{0}}{t^{2} \over (1+at)^{3}}.
\end{align*}
In the similar way we prove the estimates on the curve $\Gamma_{3}(r)=re^{-i\varphi_{0}}$. 

Let $\Gamma_{2}(\varphi)=t^{-1}e^{i\varphi}$ for $-\varphi_{0} \leq \varphi \leq \varphi_{0}$. Then we have
\begin{align*}
|\int_{\Gamma_{2}} e^{zt} (a+|z|)^{-3} \ dz| 
&= |\int_{-\varphi_{0}}^{\varphi_{0}} e^{tt^{-1}e^{i\varphi}}t^{-1}ie^{i\varphi}(a+t^{-1})^{-3} d\varphi|
\\
&\leq {t^{2} \over (1+at)^{-3}} \int_{-\varphi_{0}}^{\varphi_{0}} e^{\cos(\varphi)} d\varphi
= c_{\varphi_{0}}'{t^{2} \over (1+at)^{3}}.
\end{align*}
\end{proof}


Now, we restrict our attention to the Heisenberg group.
Let $\mu_{t}$ be a convolution semigroup of measure on the Heisenberg group. The group Fourier transform
\begin{equation*}
\pi^{\lambda}_{\mu_{t}} = \int_{\mathfrak{h}} (\pi^{\lambda}_{h})^{*} \mu_{t}(dh), \qquad \lambda \neq 0.
\end{equation*}
is also well-defined for probabilistic measures. Moreover,
\begin{equation*}
|\langle \pi^{\lambda}_{\mu_{t}}f,g \rangle_{\ltrn}| \leq 
\|f\|_{\ltrn}\|g\|_{\ltrn}.
\end{equation*}
Thus, for every $\lambda \neq 0$, the operators
\begin{equation*}
T_{t}^{\lambda}f=\pi^{\lambda}_{\mu_{t}}f , \qquad f \in \lth,
\end{equation*}
are defined and $\|T_{t}^{\lambda}\| \leq 1$. It is not hard to see that 
\begin{equation*}
T_{t}^{\lambda}T_{s}^{\lambda}=T_{t+s}^{\lambda}, \qquad s,t >0, \ \lambda \neq 0.
\end{equation*}

Let $P$ be the generating functional of the semigroup $\mu_{t}$. $P$ is generalized laplacian and it can be decompose as a sum of distribution with compact support $P_{1}$ and bounded measure $\mu$. In particular, we define the operator
\begin{equation} \label{rozp}
\pi^{\lambda}_{P} = \pi^{\lambda}_{P_{1}}+\pi^{\lambda}_{\mu}.
\end{equation}

\begin{thm}
Let $P$ be a generalized laplacian. Then for all $\lambda \neq 0$ the operators $T^{\lambda}_{t}$ form a strongly continuous semigroup of operators on $L^{2}(\mathds{R}^{n})$. The infinitesimal generator $A^{\lambda}$ of semigroup $T_{t}^{\lambda}$ is the closure of the operator $\pi^{\lambda}_{P}$.
\end{thm}

The semigroup $T_{t}^{\lambda}$ and its infinitesimal generator $A^{\lambda}$ can be though as quantized version of the semigroup given by $U_{t}f=\mu_{t}*f$ and its infinitesimal generator $\mathcal{L}$. A precise description of quantizied generators one can find in Applebaum-Cohen \cite{ApC}.


\subsection{Admissible generalised laplacians on the Heisenberg group}
In our work we consider semigroups of measures which are generated by generalized laplacians satisfying some condition of admissibility.

Let us recall some aspects of theory of the Abelian groups $\rd$. The classical reference is Berg-Forst \cite{BeFo}.
Continuous \textit{negative definite functions} on $\rd$ are characterized by L{\'e}vy-Khintchine formula
\begin{equation*}
\psi(\xi) = c+ib \cdot \xi + \xi \cdot a\xi + \int_{\rd \backslash \{0\}} (1-e^{i\xi \cdot y}+y\cdot \xi \mathds{1}_{\|y\| \leq 1}) \nu d(y),
\end{equation*}
where $c \geq 0$, $b \in \rd$, $a$ is a positive definite measure $d \times d$, and $\nu$ is a measure integrating the function $\min(1,\|y\|^{2})$, i.e. L{\'e}vy measure.

It follows from theorems of Bochner and Schoenberg that $P$ is a generalized laplacian on $\rd$ iff $\psi=-\widehat{P}$ is a continuous negative definite function. 
Moreover, if $\nu_{t}$ is a convolution semigroup of measures on the Abelian group $\rd$ with generating functional $P$ and $\psi=-\widehat{P}$, then
\begin{equation*}
\widehat{\nu_{t}}(\xi) = e^{-t\psi(\xi)}.
\end{equation*}

Suppose that $P$ is a generalised laplacian on $\mathds{R}^{2n+1}$ and let $\psi=-\widehat{P}$.
Let us recall that $\rho(\xi)=(1+\|\xi\|^{2})^{1 \over 2}$). 

We say that a continuous negative definite function $\psi$ is \textit{admissible}, if
\begin{enumerate}[\upshape (i)]
\item $\psi$ is real,
\item $1+\psi$ is a weight for $\rho$ on the Heisenberg group and $\psi \leq C\rho$,
\item $|\partial^{\alpha}_{w}\partial^{\beta}_{\lambda} \psi(w,\lambda)| \leq c_{\alpha, \beta} (1+\psi(w,\lambda)) \rho(w,\lambda)^{-|\alpha|-\beta}$, \qquad $\alpha \in \mathds{N}^{2n}$, $\beta \in \mathds{N}$.
\end{enumerate}
Alternatively, we say that a generalized laplacin $P$ is \textit{admissible}, if the function $\psi=-\widehat{P}$ is admissible.

The conditions $(ii)$ and $(iii)$ guarantee in particular that $P$ is an $\mathcal{S}$-convolver on the Heisenberg group or the on groups $\mathfrak{h}_{\theta}$, $\theta \in [0,1]$. It follows from Fact \ref{f54}. The same property has the resolvent by Lemma \ref{spod}.
The condition $(i)$ implies that $P=\widetilde{P}$, i.e. $P$ is symmetric.

\begin{exa}
A negative definite function $\psi(\xi)=\log(1+\|\xi\|^{2})$ is admissible.
\end{exa}
\begin{proof}
The condition $(i)$ is of course satisfied. The condition $(ii)$ was checked in Example \ref{p54}. We estimate a partial derivative $\partial_{\xi_{j}}$ of considered function. We have
\begin{equation*}
|\partial_{\xi_{j}}\log(1+\|\xi\|^{2})| = |2\xi_{j}| (1+\|\xi\|^{2})^{-1} \leq 2(1+\|\xi\|^{2})^{-{1 \over 2}}.
\end{equation*}
In general, the condition $(iii)$ follows from Fa\`a di Bruno formula.
\end{proof}


We say that a smooth function $f:(0,\infty) \to \mathds{R}$ with continuous extension to $[0,\infty)$ is a Bernstein function if
\begin{equation*}
f \geq 0, \qquad (-1)^{k}f^{(k)} \leq 0, \qquad k \in \mathds{N}.
\end{equation*}
It is well-known that if $\psi$ is a continuous negative definite function and $u$ is a Bernstein function, then the composition $u \circ \psi$ is again a continuous negative definite function. Moreover, Bernstein functions have the following properties.
\begin{prop} \label{p82}
If $u$ is a Bernstein function, then
\begin{enumerate}[\upshape (1)] 
\item $u$ is increasing,
\item $|u^{(k)}(t)| \leq k!t^{-k}u(t)$, \qquad $k \in \mathds{N}$, $t>0$,
\item $g(t)=(1+u(t))t^{-1}$ is decreasing.
\end{enumerate}
\end{prop}

Let us notice that the function $\log(1+\|\xi\|^{2})$ has a form $u(\|\xi\|^{2})$, where $u$ is a Bernstein function $u(t)=\log(1+t)$.
In the above example the function of the form $\psi(\xi)=u(\|\xi\|)$ was smooth. In is not true, in general, consider e.g. $\psi(\xi)=\log(1+\|\xi\|)$.
We omit this problem by considering functions of the form $u(1+\|\xi\|^{2})$.

\begin{prop} \label{berprop}
Let $u$ be a Bernstein function and let $u$ satisfies the condition $u(t) \leq Ct^{1 \over 2}$, $t>0$. Then the function $a_{u}:\mathds{R}^{2n+1} \to \mathds{R}$ given by
\begin{equation*}
a_{u}(\xi) = u(1+\|\xi\|^{2}) = u \circ \rho^{2}(\xi)
\end{equation*}
is admissible negative definite function.
\end{prop}
\begin{proof}
By definition $a_{u}$ is real. By the condition $u(t) \leq Ct^{1 \over 2}$, $t>0$ we get that
$a_{u}(\xi) \leq C\rho(\xi)$. Moreover, we obtain the following estimations
\begin{equation*}
|\partial^{\alpha}_{\xi} a_{u}(\xi)| \leq c_{\alpha} a_{u}(\xi)\rho(\xi)^{-|\alpha|}, \alpha \in \mathds{N}^{2n+1}, \ \qquad \xi \in \mathds{R}^{2n+1}.
\end{equation*}
Indeed, by Fa\`a di Bruno formula we get
\begin{equation*}
\partial^{\alpha}_{\xi} (u\circ\rho^{2}(\xi)) = \sum_{j=1}^{|\alpha|} u^{(j)}(\rho^{2}(\xi)) \sum {\alpha! \over k_{1}!\ldots k_{p}!} \left({(\partial^{\beta_{1}}\rho^{2})(\xi) \over \beta_{1}!}\right)^{k_{1}}\ldots\left({(\partial^{\beta_{p}}\rho^{2})(\xi) \over \beta_{p}!}\right)^{k_{p}},
\end{equation*}
where $\sum_{i=1}^{p}k_{i}\beta_{i}=\alpha$ and $\sum_{i=1}^{p}k_{i}=j$.
It follows from the property $(2)$ in Proposition \ref{p82} that
\begin{equation*}
|u^{(j)}(\rho^{2}(\xi))| \leq j!u(\rho^{2}(\xi))\rho(\xi)^{-{j}}.
\end{equation*}
By $\rho^{2} \in S(\rho^{2},\rho;\mathfrak{h}^{*})$ we get 
\begin{align*}
|\partial^{\alpha}_{\xi} a_{u}(\xi)|
&\leq c_{\alpha}
\sum_{j=1}^{|\alpha|} j! \rho(\xi)^{-2j} a_{u}(\xi) \prod (\rho(\xi)^{2-|\beta_{i}|})^{k_{i}}
\\
&\leq c_{\alpha}
\sum_{j=1}^{|\alpha|} j! a_{u}(\xi) \rho(\xi)^{-2j+\sum(2-|\beta_{i}|)k_{i}}
\leq c'_{\alpha} a_{u}(\xi)\rho(\xi)^{-|\alpha|}.
\end{align*}

It is enough to show that $1+a_{u}$ is a weight for the weight function $\rho$ on $\mathfrak{h}^{*}$.
We check slowness. If $\|w\| \leq \|v\|$, then
\begin{equation*}
{1+a_{u}(w,\lambda) \over 1+a_{u}(v,\lambda)} \leq C,
\end{equation*}
as Bernstein functions are increasing.
If $\|u\| \leq \|v\|$, then, by condition $(3)$ in Proposition \ref{p82} we get that
\begin{equation*}
{1+a_{u}(w,\lambda) \over 1+a_{u}(v,\lambda)} = {1+u(1+\|w\|^{2}+|\lambda|^{2}) \over 1+u(1+\|v\|^{2}+|\lambda|^{2})}
\leq {1+\|w\|^{2}+|\lambda|^{2} \over 1+\|v\|^{2}+|\lambda|^{2}},
\end{equation*}
and condition of slowness of the function $\rho^{2}$ is satisfied.
In the similar way we get temperance.
\end{proof}

In the above way one can get the following admissible negative definite functions: $\log(2+\|\xi\|^{2})$, $(1+\|\xi\|^{2})^{\delta}$, $\delta \in (0,{1 \over 2}]$, $(1+\|\xi\|^{2}+m^{2})^{1 \over 2}-m$ for $m \geq 0$.


\subsection{The Main Theorem}
We consider a generalized laplacian $P$. Let $-\widehat{P}=\psi$. The distribution $P$ is the generating functional of the semigroup of measures $\mu_{t}$ on the Heisenberg group and semigroup of measures $\nu_{t}$ on the Abelian group $\mathds{R}^{2n+1}$. Let $U_{t}$ and $V_{t}$ be associated strongly continuous semigroups of operators
\begin{equation*}
U_{t}f=\mu_{t}*f, \qquad V_{t}f=\nu_{t}*_{0}f, \qquad f \in C_{0}(\mathds{R}^{2n+1}).
\end{equation*}
The semigroups cutting to $C_{c}^{\infty}(\mathds{R}^{2n+1})$ can be extended to semigroups of contractions on $L^{2}(\mathds{R}^{2n+1})$ and will be denoted the same.
The infinitesimal generators (perhaps unbounded) $\mathcal{A}$ and $\mathcal{A}^{0}$ (with domains) of semigroups of operators $U_{t}$ and $V_{t}$ satisfies
\begin{equation*}
\mathcal{A}f = \Op(P)f = P*f, \qquad \mathcal{A}^{0}f = \Op^{0}(P)f = P*_{0}f, \qquad f \in C_{c}^{\infty}(\mathds{R}^{2n+1}).
\end{equation*}
Symmetricity $P=\widetilde{P}$ implies that the operators $\Op(P)$ and $\Op^{0}(P)$ (as well as the semigroups $U_{t}$ and $V_{t}$) are selfadjoint on $L^{2}(\mathds{R}^{2n+1})$. Consequently, one can extend the semigroups $U_{t}$ and $V_{t}$ to analytic semigroups.
Admissibility of $P$ also implies that $P$ is an $\sco$-convolver. 

Notice that distribution $-\delta_{0}$ is a generalized laplacian and $\widehat{\delta_{0}}=1$. Its semigroups of measures has densities $e^{-t}$ on the Abelian group or the Heisenberg group.
For our convenience we will consider the distribution $A=P-\delta_{0}$. $A$ is a generalized laplacian and its Fourier transform is given by $m=1+\psi$. $A$ is the generating functional of semigroup of measures $e^{-t}\mu_{t}$ on the Heisenberg group and $e^{-t}\nu_{t}$ on the Abelian group.
By admissibility of $P$ we have that $A \in S(m,\rho;\mathfrak{h})$. Moreover, $m$ is elliptic weight for $\rho$ and $A$ is maximal operator.

The operators $\Op(A)$ and $\Op^{0}(A)$ are selfadjoint with the domain $\mathcal{D}=H(m,g;\mathfrak{h})$.
Their resolvent sets contain some sector
$\Sigma_{\psi}$ and we will restrict to them. Let $B_{z}=(zI-A)^{-1}$ be the resolvent of the operator $\Op(A)$ and $B^{0}_{z}$ the resolvent of the operator $\Op^{0}(A)$.
Also let,
\begin{equation*}
b_{z}=\widehat{B}_{z}, \qquad b_{z}^{0}=\widehat{B^{0}}_{z}=(z-a)^{-1}.
\end{equation*}
By Proposition \ref{rzlr} we get.
\begin{cor} \label{finwni} 
The symbol $b_{z}$ is in the class $\smzigh$ with seminorms independent from $z$.
Moreover,
\begin{equation*}
b_{z}=b_{z}^{0}+H_{z},
\end{equation*}
where $H_{z} \in S({m \over (m+|z|)^{2}}\Lambda^{2} \rho^{-4},g;\mathfrak{h}^{*})$ uniformly with respect to $z$.
\end{cor}


 
\begin{thm} \label{finprop}
Suppose that $P$ is an admissible generalized laplacian and let $\mu_{t}$ and $\nu_{t}$ be the semigroups generating by $P$. Let $m=-\widehat{P}+1$. Then the function
\begin{equation*} 
\widehat{\mu_{t}} - \widehat{\nu_{t}} = h_{t}
\end{equation*}
is smooth and satisfies
\begin{equation*} 
|\partial^{\alpha} h_{t}(w,\lambda)| \leq c_{\alpha} {t^{2}e^{t} \over (1+tm(w,\lambda))^{3}}|\lambda|^{2}\rho(w,\lambda)^{-4-|\alpha|}, \qquad \alpha \in \mathds{N}^{2n+1}, t>0.
\end{equation*}
\end{thm}
\begin{proof}
Let $A=P-\delta_{0}$. Then $e^{-t}\mu_{t}$ and $e^{-t}\nu_{t}$ are the semigroups generating by $A$ on the Heisenberg group and the Abelian group, respectively.
Let us use the notion as in Corollary \ref{finwni}. By Proposition \ref{classy} we get
\begin{equation*} 
e^{-t}\widehat{\mu_{t}} = \int_{\Gamma} e^{zt} b_{z} \ dz, \qquad e^{-t}\widehat{\nu_{t}} = \int_{\Gamma} e^{zt} b_{z}^{0} \ dz, \qquad z \in \Sigma_{{\pi \over 2}+\varphi}.
\end{equation*}
It follows from Corollary \ref{finwni} the following decomposition of the symbol $b_{z}$
\begin{equation*}
b_{z}=b_{z}^{0}+H_{z},
\end{equation*}
where $H_{z} \in S({m^{2} \over (m+|z|)^{3}}\Lambda^{2}\rho^{-4},\rho;\mathfrak{h}^{*})$. Thus, 
\begin{equation*} 
e^{-t}\widehat{\mu_{t}} = \int_{\Gamma} e^{zt} b_{z} \ dz,
= \int_{\Gamma} e^{zt} b_{z}^{0} \ dz 
+\int_{\Gamma} e^{zt} H_{z}  dz
= e^{-t}\widehat{\nu_{t}}
+ \int_{\Gamma} e^{zt} H_{z} dz.
\end{equation*}
By Lemma \ref{lga} we have the estimates
\begin{equation*} 
|\int_{\Gamma} e^{zt} (m(w,\lambda)+|z|)^{-3} \ dz| \leq Ct^{2}(1+tm(w,\lambda))^{-3}.
\end{equation*}
Thus,
\begin{equation*}
|\partial^{\alpha} h_{t}(w,\lambda)|=|\partial^{\alpha}e^{t}\int_{\Gamma} e^{zt}H_{z}(w,\lambda)| \leq c_{\alpha} {t^{2}e^{t} \over (1+tm(w,\lambda))^{3}}|\lambda|^{2}g(w,\lambda)^{-4-|\alpha|}, \qquad \alpha \in \mathds{N}^{2n+1}.
\end{equation*}
\end{proof}

\begin{cor}
Suppose that $P$ is an admissible generalized laplacian. Let $\mu_{t}$ and $\nu_{t}$ be the semigroups generated by $P$. Then,
\begin{equation*} 
|\partial^{\alpha}(\widehat{\mu_{t}} - \widehat{\nu_{t}})(\xi)| \leq c_{\alpha} \min(t^{2},t^{-1})e^{t} \min(\|\xi\|^{2-|\alpha|},\|\xi\|^{-2-|\alpha|}), \qquad \alpha \in \mathds{N}^{2n+1}, \ t>0, \ \xi \neq 0.
\end{equation*}
\end{cor}
\begin{proof}
It is not hard to see that
\begin{equation*} 
{t^{2} \over (1+tm(\xi))^{3}} \leq \min(t^{2},t^{-1}), \qquad t>0, \xi \in \mathds{R}^{2n+1},
\end{equation*}
which follows from $m \geq 1$.
We also have $|\xi_{2n+1}| \leq \|\xi\|$ and $\rho(\xi) \geq 1$ and $\rho(\xi) \geq \|\xi\|$. Thus,
\begin{equation*}
|\xi_{2n+1}|^{2}\rho(\xi)^{-4-|\alpha|} \leq \|\xi\|^{2-|\alpha|}.
\end{equation*}
On the other hand $|\xi_{2n+1}| \leq \rho(\xi)$ and then
\begin{equation*}
|\xi_{2n+1}|^{2}\rho(\xi)^{-4-|\alpha|} \leq \|\xi\|^{-2-|\alpha|}.
\end{equation*}
\end{proof}

\begin{thm} \label{mt}
Suppose that $P$ is an admissible generalized laplacian. Let $\mu_{t}$ and $\nu_{t}$ be semigroups of measures generated by $P$ on the Heisenberg group and the Abelian group $\mathds{R}^{2n+1}$, respectively. Then the diffeence of the measures $\mu_{t}$ and $\nu_{t}$, denoted by $k_{t}$, agrees with a smooth function outside zero and for all $N \in \mathds{N}$ and all $\alpha \in \mathds{N}^{2n+1}$ it satisfies
$$|\partial^{\alpha} k_{t}(x)| \leq \begin{cases}
 c_{n,\alpha} \min(t^{2},t^{-1})e^{t}\|x\|^{-(2n+1)+2-|\alpha|} & \mbox{dla } \|x\| \leq 1,\cr
 c_{n,\alpha,N} \min(t^{2},t^{-1})e^{t} \|x\|^{-N} & \mbox{dla } \|x\| \geq 1. \end{cases}$$
\end{thm}
\begin{proof}
By the Fourier transform and Theorem \ref{finprop} we obtain
\begin{equation*} 
\mu_{t} - \nu_{t} = h_{t}^{\vee}=k_{t}.
\end{equation*}
The estimates
\begin{equation*} 
|\partial^{\alpha}h_{t}(\xi)| \leq c_{\alpha}' \min(t^{2},t^{-1})e^{t} \|\xi\|^{-2-|\alpha|}, \qquad \alpha \in \mathds{N}^{2n+1}, t>0.
\end{equation*}
implies by Corollary \ref{st0thm} that
\begin{equation*} 
|\partial^{\alpha} k_{t}(x)| \leq c_{\alpha} \min(t^{2},t^{-1})e^{t} \|x\|^{-(2n+1)+2-|\alpha|}.
\end{equation*}
On the other hand, by Theorem \ref{finprop}, the difference $h_{t}$ of the Fourier transforms of the measures $\mu_{t}$ and $\nu_{t}$ is in the class $S(\rho^{-2},\rho;\mathfrak{h}^{*})$. By the similar reasoning as in Fact \ref{f54} we get that for all $N \in \mathds{N}$ and all $\alpha \in \mathds{N}^{2n+1}$, \begin{equation*} 
|\partial^{\alpha} k_{t}(x)| \leq c_{\alpha,N} \min(t^{2},t^{-1})e^{t} \|x\|^{-N}, \qquad \|x\| \geq 1.
\end{equation*}  
\end{proof}


\subsection{Application of the main theorem} \label{sec9}



Directly from Theorem \ref{mt} we get that the funcion $k_{t}$ (the difference of the semigroups of measures) belongs to $L^{p}$ iff ${2p \over p-1} > 2n+1$. In particular we obtain the following corollary.
\begin{cor} \label{wmt}
Suppose that $P$ is an admissible generalized laplacian and let $\mu_{t}$ and $\nu_{t}$ be the semigroups of measures on the Heisenberg group and the Abelian group, respectively, generated by $P$. Then, the condition ${2p \over p-1} > 2n+1$ implies that the measures $\mu_{t}$ belongs to $L^{p}$ iff $\nu_{t}$ belongs to $L^{p}$. In particular,
\begin{itemize}
\item the measures $\mu_{t}$ have densities in $L^{1}$ iff the measures $\nu_{t}$ have densities in $L^{1}$,
\item for $n=1$, i.e. $\mathfrak{h} \cong \mathds{R}^{3}$, the measures $\mu_{t}$ are in $L^{2}$ iff the measures $\nu_{t}$ are in $L^{2}$.
\end{itemize}
\end{cor}

%

Let us recall that the function given by
\begin{eqnarray} \label{fuka}
K_{u}(v) = {1 \over 2} \int_{0}^{\infty} s^{-u-1} e^{-{v \over 2} (s + s^{-1})} \ ds, \qquad u \in \mathds{R}, \ v >0,
\end{eqnarray}
is the modified Bessel function of the second kind (or McDonald function).

Now, we consider two examples of semigroup of measures. We apply Theorem \ref{mt} to get pointwise estimates for their densities.

Let $\gamma_{t}$ be Gamma distribution. 
Then, the measures with densities $\breve{\gamma_{t}}=\gamma_{t}*_{0}\widetilde{\gamma_{t}}$ form a semigroup with respect to the convolution $*_{0}$ which is called symmetric gamma distribution or gamma-variance semigroup. Its generating functional is the distribution
\begin{equation*}
\langle \breve{\Gamma}, f \rangle = \int_{\mathds{R} \backslash \{0\}} {f(x)-f(0) \over |x|} e^{-|x|} \ dx, \qquad f \in \mathcal{S}(\mathds{R}).
\end{equation*}
Its Fourier transform is given by $-\log(1+|\xi|^2)$, $\xi \in \mathds{R}$.
Its $d$-dimensional analogue is the generalized laplacian
\begin{equation*}
\langle \Gamma, f \rangle = c_{d} \int_{\rd \backslash \{0\}} {f(x)-f(0) \over \|x\|^{{d \over 2}}} K_{{d \over 2}}(\|x\|) dx, \qquad f \in \mathcal{S}(\rd).
\end{equation*}
The Fourier transform is given by
\begin{equation*}
\varphi(\xi)=-\log(1+\|\xi\|^2), \qquad \xi \in \rd.
\end{equation*}
By the form of $\widehat{v_{t}}$
\begin{equation*}
\widehat{v_{t}}(\xi)=e^{-t\log(1+\|\xi\|^{2})}=(1+\|\xi\|^{2})^{-t}
\end{equation*}
we get that the densities of the semigroup generated by the functional $\Gamma$ are given by
\begin{equation} \label{gaden}
v_{t}(x)=c_{t,d}\|x\|^{t-{d \over 2}}K_{t-{d \over 2}}(\|x\|).
\end{equation}
In particular, if $\|x\| \to \infty$, then
\begin{equation*}
p_{t}(x) \asymp c_{t}\|x\|^{t-1}e^{-\|x\|}.
\end{equation*}
The formula (\ref{gaden}) aldo implies that for $0<t<{2n+1 \over 2}$ and $\|x\| \to 0$ we have
\begin{equation*}
p_{t}(x) \asymp c_{t}\|x\|^{2t-(2n+1)}.
\end{equation*}


Let us consider the semigroup of measures on the Heisenberg group $\mathfrak{h}$ with the functional
\begin{equation*}
\langle \Gamma, f \rangle = \lim_{\epsilon \to 0} \int_{\|x\| \geq \epsilon} {f(x)-f(0) \over \|x\|^{{2n+1 \over 2}}} K_{{2n+1 \over 2}}(\|x\|) \ dx.
\end{equation*}
We checked in Example \ref{p54} that it is an admissible generalized laplacian.

\begin{cor} \label{fincor}
The densities of the semigroup of measures on the Heisenberg group with the generating functional $\Gamma$ satisfy
\begin{equation*}
q_{t}(x) \leq ct\|x\|^{2t-(2n+1)}, \qquad t<1, \|x\| \leq 1.
\end{equation*}
Moreover, we have the following asymptotic behavior
\begin{equation*}
q_{t}(x) \asymp t\|x\|^{2t-(2n+1)}, \qquad t<1, \ \|x\| \to 0.
\end{equation*}
\end{cor}
\begin{proof}
By (\ref{gaden}),
\begin{equation*}
p_{t}(x) \leq ct\|x\|^{2t-(2n+1)}, \qquad t<1, \ \|x\| \leq 1.
\end{equation*}
Moreover, we have
\begin{equation*}
p_{t}(x) \asymp t\|x\|^{2t-(2n+1)}, \qquad \|x\| \to 0, \ t \to 0.
\end{equation*}

By Theorem \ref{mt} we get that for $\|x\| \leq 1$ and $t \leq 1$ we have
\begin{equation*} 
|(p_{t}-q_{t})(x)| \leq c t^{2}\|x\|^{-(2n+1)+2}.
\end{equation*}
Consequently, for small time $t$ we obtain the estimates in the case $\|x\| \leq 1$ or asymptotic behavior.
\end{proof}


For $\alpha \in (0,2)$ and $m>0$ we consider the Bernstein functions $u(t)=(m^{2 \over \alpha}+t)^{\alpha \over 2}-m$. Then, the negative definite functions 
$(m^{2 \over \alpha}+\|\xi\|^{2})^{\alpha \over 2}-m$ provide to relativistic $\alpha$-stable semigroups, whose Fourier transforms satisfy
\begin{equation*}
\widehat{\rho_{t}^{\alpha}}(\xi) = e^{-t((m^{2 \over \alpha}+\|\xi\|^{2})^{\alpha \over 2}-m)}.
\end{equation*}
It is known (see e.g. Kulczycki-Siudeja \cite{KuS}) that the densities of the semigroup satisfy for some constants $c_{1}$, $c_{2}$,
\begin{equation*}
\rho_{t}^{\alpha}(x) \leq c_{1}e^{tm}\min(t\|x\|^{-d-\alpha}e^{-c_{2}\|x\|},t^{-{d \over \alpha}}).
\end{equation*}
In particular, for $\|x\| \leq 1$ we have
\begin{equation*}
\rho_{t}^{\alpha}(x) \leq c t\|x\|^{-d-\alpha}, \qquad t \to 0.
\end{equation*}


\begin{lem}
Let $\alpha \in (0,1]$. Then, the continuous negative definite function
\begin{equation*}
\psi(\xi)=(m^{2 \over \alpha}+\|\xi\|^{2})^{\alpha \over 2}-m,
\end{equation*}
is admissible for some $m>0$. In particular, it is admissible for $m=1$ and $\alpha \in (0,1]$ and moreover for $m$ and $\alpha$ satisfy $({\alpha \over 2})^{\alpha \over 2}< m <1$. 
\end{lem}
\begin{proof}
It is not hard to check the case $m=1$. We focus on the case $({\alpha \over 2})^{\alpha \over 2}< m <1$.

The behavior of the derivatives we get in the similar way as in Proposition \ref{berprop}. We check that $\psi$ is a weight for $\rho$.
For $\alpha \in (0,1]$ we have $\psi(\xi) \leq \rho(\xi)$ (for $\alpha \in (1,2]$ this condition is not satisfied). Let us notice that the function $\psi$ is increasing. We show that $1+\psi(\xi) \over 1+\|\xi\|^{2}$ is decreasing, what will be the end. As $\psi(\xi)=u(\|\xi\|^{2})$, we check that ${1+u(t) \over 1+t}$ is decreasing. For $m \leq 1$ we have
\begin{align*}
\left({1+u(t) \over 1+t}\right)'
&={{\alpha \over 2}(m^{2 \over \alpha}+t)^{{\alpha \over 2}-1}(1+t)-(1+(m^{2 \over \alpha}+t)^{\alpha \over 2}-m) \over (1+t)^{2}}
\\
&\leq {{\alpha \over 2}(1+t)-(1-m)(m^{2 \over \alpha}+t)^{{\alpha \over 2}-1}-(m^{2 \over \alpha}+t) \over (m^{2 \over \alpha}+t)^{1-{\alpha \over 2}}(1+t)^{2}}
\leq {{\alpha \over 2}-m^{2 \over \alpha} \over (m^{2 \over \alpha}+t)^{1-{\alpha \over 2}}(1+t)^{2}},
\end{align*}
and thus, if in addition $({\alpha \over 2})^{\alpha \over 2}<m$, then ${\alpha \over 2} < m^{{2 \over \alpha}}$ and the above derivative is negative.
\end{proof}


Let $v_{t}$ be the densities of the convolution semigroup of measures with the associated negative definite function of the form $(m^{2 \over \alpha}+\|\xi\|^{2})^{\alpha \over 2}-m$.
By Theorem \ref{mt} we get that for $\|x\| \leq 1$ and $t \leq 1$ we have
\begin{equation*} 
|(\rho_{t}-v_{t})(x)| \leq c t^{2}\|x\|^{-(2n+1)+2}.
\end{equation*}
In the similar way as in Corollary \ref{fincor} we get the estimates on the Heisenberg group.
\begin{cor}
Let $v_{t}$ be the densities of the convolution semigroup of measures with the associated negative definite function of the form $(m^{2 \over \alpha}+\|\xi\|^{2})^{\alpha \over 2}-m$. Then,
\begin{equation*}
v_{t}(x) \leq ct \|x\|^{-(2n+1)-\alpha}, \qquad t<1, \ \|x\| \leq 1.
\end{equation*}
\end{cor}



\subsection*{Acknowledgements.}
The author is grateful to P. G{\l}owacki for inspiring conversations on the subject of the present note and his useful suggestions. He also thanks M. Mirek for critical reading the manuscript and M. Preisner for his helpful remarks.


\begin{thebibliography}{HD}

\normalsize\baselineskip=17pt


\bibitem{App} D. Appelbaum
\emph{Some L2 properties of semigroups of measure on Lie groups},
Semigroup Forum 79 (2009), 217-228.

\bibitem{ApC} D. Appelbaum, S. Cohen
\emph{L{\'e}vy processes, pseudo-differential operators and Dirichlet form in the Heisenberg group},
Ann. de la Faculte des Sci. de Toulouse 13 (2004), 149-177.

\bibitem{BKG} H. Bahouri, C. Fermanian-Kammerer, I. Gallagher
\emph{Phase-space analysis and pseudodifferential calculus on the Heisenberg group},
Asterisque, (342):vi+127, 2012.

\bibitem{BaP} M. Barczy, G. Pap
\emph{Fourier transform of a Gaussian measure on the Heisenberg group},
Ann. de l'Ins. H. Poincare (B) Prob. and Stat. 42 (2006), 607-633.

\bibitem{Be1} R. Beals
\emph{General calculus of pseudodifferential operators},
Duke Math. J. 42 (1975), 1-42.

\bibitem{Be2} R. Beals
\emph{Characterization of pseudodifferential operators and applications},
Duke Math. J. 44 (1977), 45-57 .
\emph{Correction},
Duke Math. J. 46 (1979), 215.

\bibitem{Be3} R. Beals
\emph{Weighted distribution spaces and pseudodifferential operators},
Journal d'Analyse Math{\'e}matique, 39 (1981), 131-187.

\bibitem{KaBe} K. Beka{\l}a
\emph{Leibniz's rule on two-step nilpotent Lie groups},
Colloquium Mathematicum 145 (2016), 137-148.

\bibitem{BeFo} C. Berg, G. Forst
\emph{Potential theory on locally compact Abelian groups},
Springer Verlag, Berlin Heidelberg New York, 1975.

\bibitem{blu} A. Bonfiglioli, E. Lanconelli, F. Uguzzoni
\emph{Stratified Lie Groups and Potetntial Theory for their Sub-Laplacians},
Springer Monographs in Mathematic, Springer-Verlag, Berlin Heidelberg, 2007.

\bibitem{BC} J-M. Bony, J-Y. Chemin 
\emph{Espaces fonctionnels associ{\'e}s au calcul de Weyl-H{\"o}rmander}, Bulletin de la Soci{\'e}t{\'e} Math{\'e}matique de France, 122/1 (1994), 77-118.

\bibitem{JMB2} J.-M. Bony
\emph{Caracterisation des operateurs pseudo-differentials},
Seminaire Equations aux Derivees Partielles (1996-1997), Ecole Polytechnique.

\bibitem{JMB} J.-M. Bony
\emph{On the Characterization of Pseudodifferential Operators (old and new)},
Studies in Phase Space Analysis with Applications to PDES, Springer, New York 2013.

\bibitem{Cor} L. Corwin
\emph{Tempered Distributions on Heisenberg Groups Whose Convolution with Schwartz Class Functions Is Schwartz Class},
Journal of Functional Analysis 44 (1981), 328-347.

\bibitem{CoGr} L. Corwin, F.P. Greenleaf
\emph{Representations of nilpotent Lie groups and their applications. Part 1: Basic theory and examples},
Cambridge University Press, Cambridge 1990.


\bibitem{DiMa} J. Dixmier, P. Malliavin
\emph{Factorisations de fonctions et de vecteurs ind{\'e}finiment diff{\'e}rentiables},
Bull. Sci. Math. 102 (1978), 307–330.

\bibitem{Duf} M. Duflo
\emph{Représentations de semi-groupes de mesures sur un groupe localement compact},
Ann. Inst. Fourier 28 (1978), 225-249.

\bibitem{Dzi} J. Dziuba{\'n}ski
\emph{Asymptotic behaviour of densities of stable semigroups of measures},
Prob. Theory and Rel. Fields. 87, (1991), 459-467.

\bibitem{Far} J. Faraut
\emph{Semi-groupes de mesures complexes et calcul symbolique sur les generateurs infinitesimaux de semigroupes d'operateurs},
Ann. Inst. Fourier 20 (1975), 235-301.

\bibitem{Fas} E.W. Farkas
\emph{Function spaces of generalised smoothness and pseudo-differential operators associated to a continuous negative definite functions},
Habilitationsschrift, Munchen, 2002.

\bibitem{FiRu} V. Fischer, M. Ruzhansky
\emph{Quantization on Nilpotent Lie Groups},
Birkhauser, Progress in Mathematics 314, 2016.

\bibitem{Fol} G.B. Folland
\emph{Harmonic analysis in Phase Space},
Ann. of Math. Studies, Princeton University Press, Princeton, 1989.

\bibitem{Vol} G.B. Folland
\emph{Subelliptic estimates and function spaces on nilpotent Lie groups},
Ark. Mat. 13 (1975) 161-207.

\bibitem{Gav} B. Gaveau
\emph{Holonomie stochastique et representations du groupe d'Heisenberg},
C.R. Acad. Sc. Paris 280 (1975), 571-573.

\bibitem{Glo1} P. G{\l}owacki
\emph{A symbolic calculus and $L^{2}$-boundeness on nilpotent Lie groups},
J. Fun. Anal. 206 (2004), 233-251.

\bibitem{Glo6} P. G{\l}owacki
\emph{Invertibility of convolution operators on homogeneous groups},
Rev. Mat. Iberoam. 28 (2012), 141-156.

\bibitem{Glo2} P. G{\l}owacki
\emph{Stable semi-groups of measures on the Heisenberg group},
Studia Mathematica 79 (1984), 105-138.

\bibitem{Glo3} P. G{\l}owacki
\emph{The algebra of Calder{\'o}n-Zygmund Kernels on a Homogeneous Group is inverse-closed},
Journal d'Analyse Math{\'e}matique 131 (2017), 337-365.

\bibitem{Glo5} P. G{\l}owacki
\emph{The Melin calculus for general homogeneous groups},
Ark. Mat. 45 (2007), 31-48.

\bibitem{Glo7} P. G{\l}owacki
\emph{Convergence of Semigroup of Complex Measures on a Lie Group},
J. Theor. Probability 26 (2013), 58-71.

\bibitem{GH} P. G{\l}owacki, W. Hebisch
\emph{Pointwise estimates for densities of stable semigroup of measures},
Studia mathematica 104 (1993), 243-258.


  


\bibitem{Hol} R. Howe
\emph{On the role of the Heisenberg group in harmonic analysis},
Bull. Amer. Math. Soc. 3 (1980), 821-843.

\bibitem{Hol2} R. Howe
\emph{On a connection between nilpotent Lie groups and oscillatory integrals associated to singularities},
Pac. J. Math. 73 (1977), 329-364.

\bibitem{Hol3} R. Howe
\emph{A symbolic calculus calculus for nilpotent groups},
Operator algebras and group representations, Vol. I (Neptun, 1980), volume 17
of Monogr. Stud. Math., pages 254–277. Pitman, Boston, 1984.

\bibitem{Hor} L. H{\"o}rmander
\emph{The Weyl calculus of pseudodifferential operators},
Comm. Pure Appl. Math. 32 (1979), 359-443.

\bibitem{Hul} A. Hulanicki
\emph{A class of convolution semigroups of measures on a Lie group},
Lect. Notes in Math 828 (1980), 82-101.

\bibitem{Hul2} A. Hulanicki
\emph{The distribution of energy in the Brownian motion in the Gaussian field and analytic-hypoellipticy of certain subelliptic operators on the Heisenberg group},
Studia Math. 56 (1976), 165-173.

\bibitem{Hun} G.A. Hunt
\emph{Semigroups of measures on Lie groups},
Trans. Am. Math. Soc. 81 (1956), 264-293.



\bibitem{KuS} T. Kulczycki, B. Siudeja
\emph{Intrinsic ultracontractivity of the Feynmann-Kac semigroup for relativistic stable processes},
Trans. of Amer. Soc 358 (2006), 5025-5057.

\bibitem{Ler} N. Lerner
\emph{Metrics on the phase space and non-selfadjoint pseudodifferential operators},
Birkhauser Verlag, Basel, 2010.


\bibitem{Mel} A. Melin
\emph{Parametrix constructions for right-invariant differential operators on nilpotent Lie groups},
Ann. Global Anal. Geom. 1 (1983), 79-130.


\bibitem{Paz} A. Pazy
\emph{Semigroups of Linear Operators and Applications to Partial Differential Equations},
Springer-Verlang, New York, 1983.


\bibitem{Ste} E.M. Stein
\emph{Harmonic Analysis. Real-Variable Methods},
Princeton University Press, Princeton, New Jersey, 1993.

\bibitem{Tay} M. Taylor
\emph{Noncommutative microlocal harmonic analysis I},
Mem. Amer. Math. Soc., Providence 1986. (Nowa wersja:
http://math.unc.edu/Faculty/met/ncmlms.pdf) 

\bibitem{Tha} S. Thangavelu
\emph{Harmonic analysis on the Heisenberg group},
Progress in Mathematics Vol 159, Birkh{\"a}user, Boston, 1998.


\end{thebibliography}
\end{document}